\documentclass[12pt]{amsart}
\usepackage{amssymb,latexsym,amsmath,amscd,amsthm,graphicx, color}
\usepackage[all]{xy}
\usepackage{pgf,tikz}
\usepackage{mathrsfs}
\usetikzlibrary{arrows}
\definecolor{uuuuuu}{rgb}{0.26666666666666666,0.26666666666666666,0.26666666666666666}
\definecolor{xdxdff}{rgb}{0.49019607843137253,0.49019607843137253,1.}
\definecolor{ffqqqq}{rgb}{1.,0.,0.}

\definecolor{uuuuuu}{rgb}{0.26666666666666666,0.26666666666666666,0.26666666666666666}
\definecolor{qqwuqq}{rgb}{0.,0.39215686274509803,0.}
\definecolor{zzttqq}{rgb}{0.6,0.2,0.}
\definecolor{xdxdff}{rgb}{0.49019607843137253,0.49019607843137253,1.}
\definecolor{qqqqff}{rgb}{0.,0.,1.}
\definecolor{cqcqcq}{rgb}{0.7529411764705882,0.7529411764705882,0.7529411764705882}
\definecolor{sqsqsq}{rgb}{0.12549019607843137,0.12549019607843137,0.12549019607843137}

\theoremstyle{plain}

\newtheorem{theorem}[subsection]{Theorem}

\newtheorem{lemma}[subsection]{Lemma}

\newtheorem{defi}[subsection]{Definition}
\newtheorem{prop}[subsection]{Proposition}

\theoremstyle{definition}

\newtheorem{remark}[subsection]{Remark}

\newcommand{\tarc}{\mbox{\large$\frown$}}
\newcommand{\arc}[1]{\stackrel{\tarc}{#1}}

\newtheorem{conj}[section]{Conjecture}

\newcommand{\uu}{\cup}% union
\newcommand{\UU}{\bigcup}% big union
\newcommand{\es}{\emptyset}% the empty set
\newcommand{\set}[1]{\{#1\}}% set
\newcommand{\ga}{\alpha}
\newcommand{\gb}{\beta}

\newcommand{\gd}{\delta}
\renewcommand{\gg}{\gamma}% old use >>

\newcommand{\gq}{\theta}

\newcommand{\gw}{\widehat}

\newcommand{\tit}{\textit}% text italic

\newcommand{\D}[1]{\mathbb{#1}}% Doubled - blackboard bold - only caps, uas as \D{A}
\newcommand{\te}{\text}% same as \mathrm command.

\newcommand{\pa}{\partial}

\newcommand{\tl}{\tilde}

\begin{document}
To appear, Journal of Optimization Theory and Applications
\title{Quantization for uniform distributions on hexagonal, semicircular, and elliptical curves}

\author{$^1$Gabriela Pena}
\address{School of Mathematical and Statistical Sciences\\
University of Texas Rio Grande Valley\\
1201 West University Drive\\
Edinburg, TX 78539-2999, USA.}
%\email{$^1$~gabriela.pena01@utrgv.edu, $^2$~hansapani.rodrigo@utrgv.edu, \newline $^3$~mrinal.roychowdhury@utrgv.edu, $^4$~josef.sifuentes@utrgv.edu, $^5$~erwin.suazo@utrgv.edu}
\email{$^1$gabriela.penamtz@gmail.com}
\email{\{$^2$hansapani.rodrigo,$^3$mrinal.roychowdhury, $^4$josef.sifuentes, $^5$erwin.suazo\}\newline @utrgv.edu}

\author{$^2$Hansapani Rodrigo}
 \author{$^3$Mrinal Kanti Roychowdhury}
\author{$^4$Josef Sifuentes}
 \author{$^5$Erwin Suazo}


\date{}
\maketitle

\pagestyle{myheadings}\markboth{G. Pena, H. Rodrigo, M.K. Roychowdhury, J. Sifuentes, and E. Suazo}
{Quantization for uniform distributions on hexagonal, semicircular, and elliptical curves}

\begin{abstract}
In this paper, first we have defined a uniform distribution on the boundary of a regular hexagon, and then investigated the optimal sets of $n$-means and the $n$th quantization errors for all positive integers $n$. We give an exact formula to determine them, if $n$ is of the form $n=6k$ for some positive integer $k$. We further calculate the quantization dimension, the quantization coefficient, and show that the quantization dimension is equal to the dimension of the object, and the quantization coefficient exists as a finite positive number. Then, we define a mixture of two uniform distributions on the boundary of a semicircular disc, and obtain a sequence and an algorithm, with the help of which we determine the optimal sets of $n$-means and the $n$th quantization errors for all positive integers $n$ with respect to the mixed distribution.  Finally, for a uniform distribution defined on an elliptical curve, we investigate the optimal sets of $n$-means and the $n$th quantization errors for all positive integers $n$.
\end{abstract}

\section{Introduction}

Quantization is a process of approximation with broad application in engineering and technology (see \cite{GG, GN, Z}).
For the mathematical treatment of quantization one is referred to Graf-Luschgy's book (see \cite{GL1}).
Recently, optimal quantization for uniform distributions on different regions have been investigated by several authors, for example, see \cite{DR, R, RR1, RR2}.
On the contrary, optimal quantization for uniform distributions on curves has not yet been much investigated. Such kind of problems has rigorous applications in many areas including signal processing. In this note we would like to list two such applications. The first application comes within the area of signal processing. When we drive long distances, quite often cellular signals get cut off. This happens because we are either far away from the tower, or there is no tower nearby to catch the signal. In optimal quantization one of our goals is to find the exact locations of the towers so that while driving we can get the
best signal at our cell phones. The second application comes within the area of agriculture. The amount of agricultural water usage needs to be controlled by placing a minimal number of water sprinkles (or any other resources) in a way that they cover the whole cropland as required. We can use our knowledge in quantization to place the optimal number of sprinkles in order to fully cover the cropland area as these crops are usually planted on a boundary line of a $k$-sided polygon including the shapes that we discuss in this paper. In this regard, we find that it is important to investigate the optimal sets of $n$-means and the $n$th quantization errors for the points on a boundary of a hexagonal, a semicircular, or an elliptical curve with respect to a probability distribution.

In this paper, Section~\ref{sec1} deals with the quantization for a uniform distribution defined on the boundary of a regular hexagon.  For this uniform distribution, in Theorem~\ref{Th41}, we give an exact formula to determine the optimal sets of $n$-means and the $n$th quantization errors for all $n$ if $n$ is of the form $n=6k$ for some positive integer $k$. We further calculate the quantization dimension, and the quantization coefficient, and show that the quantization dimension is equal to the dimension of the object, and the quantization coefficient exists as a finite positive number. In Proposition~\ref{prop000}, we show that for sufficiently large $n$ the points in an optimal set of $n$-means lie on the boundary of the hexagon. Section~\ref{sec2} deals with the quantization for a mixture of two uniform distributions defined on the boundary of a semicircular disc. In this section, first we have explicitly determined the optimal sets of $n$-means and the $n$th quantization errors for all $1\leq n\leq 9$. Then, we have proved Theorem~\ref{Th61}, and defined a sequence and algorithm. With the help of the sequence and the algorithm, Theorem~\ref{Th61} gives all the optimal sets of $n$-means and the $n$th quantization errors for all $n\geq 4$. Section~\ref{sec3} deals with the quantization for a uniform distribution defined on the boundary of an ellipse. In Proposition~\ref{prop70},  Proposition~\ref{prop71}, and Proposition~\ref{prop72}, we determine the optimal sets of $n$-means and the $n$th quantization errors for $n=2, 6, 7$. Following the technique given in Proposition~\ref{prop71} and Proposition~\ref{prop72}, we can obtain the optimal sets of $n$-means and the $n$th quantization errors for the uniform distribution for any positive integer $n$.

Finally, we would like to mention that the techniques, given in this paper, to determine the optimal sets of $n$-means and the $n$th quantization errors with respect to the probability distributions defined on three different curves will help to further investigate them for a more general curve. Such a problem with respect to any probability distribution, even for a uniform distribution, defined on any curve is not known yet.

\section{Preliminaries}

Let $\D R^d$ denote the $d$-dimensional Euclidean space, $\|\cdot\|$ denote the Euclidean norm on $\D R^d$ for any $d\geq 1$, and $n\in \D N$. Then, the $n$th \textit{quantization
error} for a Borel probability measure $P$ on $\D R^d$ is defined by
\begin{equation*} \label{eq0} V_n:=V_n(P)=\inf \Big\{\int \min_{a\in\alpha} \|x-a\|^2 dP(x) : \alpha \subset \mathbb R^d, \text{ card}(\alpha) \leq n \Big\}.\end{equation*}
A set $\ga$ for which the infimum is achieved and contains no more than $n$ points is called an \textit{optimal set of $n$-means} for $P$, and the points in an optimal set are called \tit{optimal quantizers}. Of course, this makes sense only if the mean squared error or the expected squared Euclidean distance $\int \| x\|^2 dP(x)$ is finite (see \cite{AW, GKL, GL1, GL2}). It is known that for a continuous probability measure an optimal set of $n$-means always has exactly $n$-elements (see \cite{GL1}). The number
\[D(P):=\lim_{n\to \infty}  \frac{2\log n}{-\log V_n(P)},\] if it exists, is called the \tit{quantization dimension} of the probability measure $P$; on the other hand, for any $s\in (0, +\infty)$, the number $\lim_{n\to \infty} n^{\frac 2 s} V_n(P)$, if it exists, is called the $s$-dimensional \tit{quantization coefficient} for $P$. For more details about the quantization dimension and the quantization coefficient, and their connections, one is referred to \cite{GL1, P}. To know about the asymptotic quantization for probability measures on Riemannian manifolds, one can see \cite{I}.

\section{Quantization for a uniform distribution on the boundary of a regular hexagon} \label{sec1}

In this section, first we give the following proposition (see \cite{GG, GL1}):
\begin{prop} \label{prop0}
Let $\ga$ be an optimal set of $n$-means for $P$, and $a\in \ga$. Then,

$(i)$ $P(M(a|\ga))>0$, $(ii)$ $ P(\partial M(a|\ga))=0$, $(iii)$ $a=E(X : X \in M(a|\ga))$,
where $M(a|\ga)$ is the Voronoi region of $a\in \ga , $ i.e.,  $M(a|\ga)$ is the set of all elements $x$ in $\D R^d$ which are closest to $a$ among all the elements in $\ga$.
\end{prop}
Due to the above proposition, we see that if $\ga$ is an optimal set and $a\in\ga$, then $a$ is the \tit{conditional expectation} of the random variable $X$ given that $X$ takes values in the Voronoi region of $a$.
Let $i$ and $j$ be the unit vectors in the positive directions of the $x_1$- and $x_2$-axes, respectively. In the sequel, the position vectors of the points $A_0, A_1, \cdots, P, Q, R, D, E, F,$ etc. will, respectively, be denoted by $\tl a_0, \tl a_1, \cdots, \tl p, \tl q, \tl r, \tl d, \tl e, \tl f$, etc, where by the position vector $\tilde a$ of a point $A$, it is meant that $\overrightarrow{OA}=\tilde a$. In addition, we will identify the position vector of a point $(a_1, a_2)$ by $(a_1, a_2):=a_1 i +a_2 j$, and apologize for any abuse in notation. For any two vectors $\vec u$ and $\vec v$, let $\vec u \cdot \vec v$ denote the dot product between the two vectors $\vec u$ and $\vec v$. Then, for any vector $\vec v$, by $(\vec v)^2$, we mean $(\vec v)^2:= \vec v\cdot \vec v$. Thus, $|\vec v|:=\sqrt{\vec v\cdot \vec v}$, which is called the length of the vector $\vec v$. For any two position vectors $\tilde a:=( a_1, a_2)$ and $\tilde b:=( b_1, b_2)$, we write $\rho(\tilde a, \tilde b):=\|(a_1, b_1)-(a_2, b_2)\|^2=(a_1-a_2)^2 +(b_1-b_2)^2$, which gives the squared Euclidean distance between the two points $(a_1, a_2)$ and $(b_1, b_2)$.  Let $P$ and $Q$ belong to an optimal set of $n$-means for some positive integer $n$, and let $D$ be a point on the boundary of the Voronoi regions of the points $P$ and $Q$. Since the boundary of the Voronoi regions of any two points is the perpendicular bisector of the line segment joining the points, we have
$|\overrightarrow{DP}|=|\overrightarrow{DQ}|, \te{ i.e., } (\overrightarrow{DP})^2=(\overrightarrow{DQ})^2$ implying
$(\tilde p-\tilde d)^2=(\tilde q-\tilde d)^2$, i.e., $\rho(\tilde d, \tilde p)-\rho(\tilde d, \tilde q)=0$. We call such an equation a \tit{canonical equation}.
By $E(X)$ and $V:=V(X)$, we represent the expectation and the variance of a random variable $X$ with respect to the probability distribution under consideration.

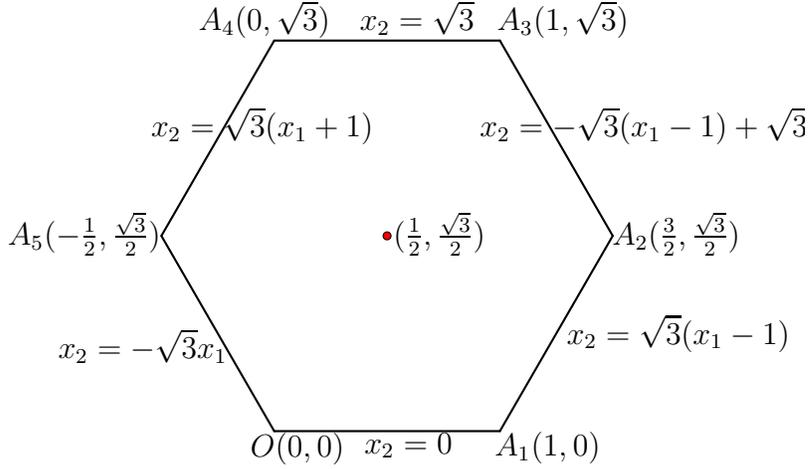
\begin{figure}
\begin{tikzpicture}[line cap=round,line join=round,>=triangle 45,x=0.5cm,y=0.5cm]
\clip(-7.2681259889226,-1.489754228607887) rectangle (14.101073814442522,12.197079045737732);
\draw [line width=0.8pt] (0.,0.)-- (6.,0.);
\draw [line width=0.8pt] (6.,0.)-- (9.,5.196152422706632);
\draw [line width=0.8pt] (9.,5.196152422706632)-- (6.,10.392304845413264);
\draw [line width=0.8pt] (0.,10.392304845413264)-- (-3.,5.196152422706632);
\draw [line width=0.8pt] (-3.,5.196152422706632)-- (0.,0.);
\draw [line width=0.8pt] (0.,10.392304845413264)-- (6.,10.392304845413264);
\draw (-0.9237452612829172,0.24761843609545116) node[anchor=north west] {$O(0, 0)$};
\draw (5.581218553862171,0.26761843609545116) node[anchor=north west] {$A_1(1, 0)$};
\draw (8.699101936361089,6.101355977887573) node[anchor=north west] {$A_2(\frac  32, \frac{\sqrt 3}{2})$};
\draw (5.641831690538654,11.729195663238811) node[anchor=north west] {$A_3(1, \sqrt 3)$};
\draw (-2.293321937782701,11.729195663238811) node[anchor=north west] {$A_4(0, \sqrt 3)$};
\draw (-7.365030256369105,6.101355977887573) node[anchor=north west] {$A_5(-\frac 12, \frac {\sqrt 3} 2)$};
\draw (7.493014266743939,3.3810200486827204) node[anchor=north west] {$x_2=\sqrt 3(x_1-1)$};
\draw (5.183014266743939,8.893954640622091) node[anchor=north west] {$x_2=-\sqrt 3(x_1-1)+\sqrt 3$};
\draw (1.989043214627868,11.729195663238811) node[anchor=north west] {$x_2=\sqrt 3$};
\draw (-3.562198083310638,8.893954640622091) node[anchor=north west] {$x_2=\sqrt 3(x_1+1)$};
\draw (-6.03200827013437,2.9871514154475536) node[anchor=north west] {$x_2=-\sqrt 3 x_1$};
\draw (2.109043214627868, 0.24761843609545116) node[anchor=north west] {$x_2=0$};
\draw (2.8816855745100686,6.101355977887573) node[anchor=north west] {$(\frac 12, \frac {\sqrt 3}2)$};
\begin{scriptsize}
\draw [fill=ffqqqq] (3.,5.196152422706632) circle (1.5pt);
\end{scriptsize}
\end{tikzpicture}
\caption{Regular hexagon with center $(\frac 12, \frac {\sqrt 3}2)$ and side length one.} \label{Fig1}
\end{figure}

Let $P$ be the uniform distribution defined on the boundary $L$ of the regular hexagon with vertices $O(0, 0), A_1(1, 0), A_2(\frac 32, \frac{\sqrt{3}}2),  A_3(1, \sqrt{3}), A_4(0, \sqrt{3}), A_5(-\frac 12, \frac{\sqrt{3}}{2})$ (see Figure~\ref{Fig1}). Let $s$ represent the distance of any point on $L$ from the origin tracing along the boundary of the hexagon in the counterclockwise direction. Then, the points $O, A_1, A_2, A_3, A_4, A_5$ are, respectively, represented by $s=0, s=1, s=2, s=3, s=4$, and $s=5$. For any two points $A$ and $B$ on the boundary $L$, not necessarily the vertices, by $\widehat{AB}$, it is meant the portion of the boundary $L$ with initial point $A$ and terminal point $B$. For example, suppose $D$ is a point on the side $OA_1$, and $E$ is a point on the side $A_2A_3$, then by $\widehat{DE}$ it is meant
\[\widehat{DE}=DA_1\uu A_1A_2\uu A_2E.\]
Notice that $\gw{ED}=EA_3\uu A_3A_4\uu A_4A_5\uu A_5O\uu OD$, and so, $\gw{DE}$ and $\gw{ED}$ are not same.
 The probability density function (pdf) $f$ of the uniform distribution $P$ is given by $f(s):=f(x_1, x_2)=\frac 1{6}$ for all $(x_1, x_2)\in L$, and zero, otherwise.  Notice that $L=\mathop{\uu}\limits_{j=1}^6 L_j$, where $L_j$, for $0\leq t\leq 1$, are represented by the parametric equations as follows:
\begin{align*}
L_j=\left\{\begin{array} {ll} \set{t \tl a_1+(1-t) \tl o } & \te{ if } j=1 \\
\set{t \tl a_j+(1-t) \tl a_{j-1}} & \te{ if } j=2, 3, 4, 5, 6.
\end{array} \right.
\end{align*}
Thus, for $0\leq t\leq 1$, we have
\begin{align*}
L_1&=\set {(t, 0)},  \quad L_2=\Big\{\Big (\frac{t}{2}+1,\frac{\sqrt{3} t}{2}\Big)\Big\}, \quad L_3=\Big\{\Big(\frac{3 (1-t)}{2}+t,\frac{1}{2} \sqrt{3} (1-t)+\sqrt{3} t\Big)\Big\},\\
L_4&=\Big\{\Big(1-t,\sqrt{3} (1-t)+\sqrt{3} t\Big)\Big\}, \quad L_5=\Big\{\Big(-\frac{t}{2},\sqrt{3} (1-t)+\frac{\sqrt{3} t}{2}\Big)\Big\}, \te{ and } \\ L_6&=\Big\{\Big(\frac{t-1}{2},\frac{1}{2} \sqrt{3} (1-t)\Big)\Big\}.
\end{align*}
Again, $dP(s)=P(ds)=f(x_1, x_2) ds=\frac 1 6 ds$. On each $L_j$ for $1\leq j\leq 6$, we have $(ds)^2=(dx_1)^2+(dx_2)^2=(dt)^2$ yielding $ds=dt$.

Let us now prove the following lemma.
\begin{lemma} \label{lemma00}
Let $X$ be a continuous random variable with uniform distribution taking values on $L$. Then,
$E(X)=(\frac 1 2, \frac {\sqrt 3}{ 2} ) \te{ and } V:=V(X)=\frac 5 6.$
\end{lemma}
\begin{proof} Recall that by $(a, b)$ it is meant $a i+b j$, where $i$ and $j$ are two unit vectors in the positive directions of $x_1$- and $x_2$-axes, respectively. Thus, we have,
\begin{align*}
E(X) &=\int_L(x_1 i+x_2 j) dP =\frac{1}{6}\Big(\int_0^1\Big( (t,0) + \Big(\frac{t}{2}+1,\frac{\sqrt{3} t}{2}\Big)+\Big(\frac{3 (1-t)}{2}+t,\frac{1}{2} \sqrt{3} (1-t)+\sqrt{3} t\Big)\\
&+\Big(1-t,\sqrt{3} (1-t)+\sqrt{3} t\Big)+\Big(-\frac{t}{2},\sqrt{3} (1-t)+\frac{\sqrt{3} t}{2}\Big)+\Big(\frac{t-1}{2},\frac{1}{2} \sqrt{3} (1-t)\Big)\Big) \, dt
\Big)
\end{align*}
implying $E(X)=(\frac{1}{2},\frac{\sqrt{3}}{2})$.  The variance $V:=V(X)$ is given by
\begin{align*} V(X)&=E\|X-E(X)\|^2= \int_L\rho((x_1, x_2), E(X))dP= \frac 1 6 \int_{L}\rho((x_1, x_2), E(X))ds\\
&=\frac 1 6 \sum_{j=1}^6\int_{L_j}\rho((x_1, x_2), E(X))ds=\int_{L_1}\rho((t, 0), (\frac{1}{2},\frac{\sqrt{3}}{2})) dt=\int_0^1\rho((t, 0), (\frac{1}{2},\frac{\sqrt{3}}{2}))\, dt=\frac 56.
\end{align*}
Hence, the proof of the lemma is complete.
\end{proof}

\begin{remark} \label{remark45} For any $(a, b) \in \D R^2$, we have
\begin{align*}
&E\|X-(a, b)\|^2=V(X)+\|(a, b)-(\frac 12, \frac {\sqrt 3}{2})\|^2,
\end{align*}
which is minimum if $(a, b)=(\frac 12, \frac {\sqrt 3}{ 2})$, and the minimum value is $V(X)$.
Thus, we see that the optimal set of one-mean is the set $\set{(\frac 12, \frac {\sqrt 3}{ 2})}$, and the corresponding quantization error is the variance $V:=V(X)$ of the random variable $X$.
\end{remark}

The following proposition gives the optimal set of two-means.

\begin{prop} \label{prop01}
 The set $\set{(\frac{13}{12},\frac{\sqrt{3}}{2}), (-\frac{1}{12},\frac{\sqrt{3}}{2})}$ forms an optimal set of two-means, and the quantization error for two-means is given by $V_2=\frac{71}{144}=0.493056$.
\end{prop}

\begin{proof}
Let $P$ and $Q$ form an optimal set of two-means. Let $\ell$ be the boundary of their Voronoi regions. Let $P$ lie in the region which contains the vertex $A_1$, and $Q$ lie in the other region. The following cases can arise:

Case~1. $\ell$ intersects the sides $OA_1$ and $A_1A_2$.

Let $\ell$ intersect $OA_1$ and $A_1A_2$ at the points $D$ and $E$, respectively. Let the parametric values for $D$ and $E$ be, respectively, $\ga$ and $\gb$, i.e., $\tilde d=(\ga, 0)$, and $\tilde e=(\frac{\gb} 2+1, \frac{\sqrt 3} 2 \gb)$, where $0\leq \ga<1$ and $0<\gb\leq 1$. Then, we have
\begin{align*}
\tilde p&=E(X : X\in \gw{DE})=\frac{\int_{\gw{DE}}(x_1, x_2) dP}{\int_{\gw{DE}}  dP}=\frac{\int_{\alpha }^1 (t,0) \, dt+ \int_0^{\beta } (\frac{t}{2}+1,\frac{\sqrt{3} t}{2})\, dt}{\int_{\alpha }^1 1 \, dt+\int_0^{\beta } 1 \, dt}\\
&=\Big(\frac{-2 \alpha ^2+\beta ^2+4 \beta +2}{-4 \alpha +4 \beta +4},\frac{\sqrt{3} \beta ^2}{-4 \alpha +4 \beta +4}\Big).
\end{align*}
Similarly,
\begin{align*}
\tilde q&=E(X : X \in \gw{ED})=\Big(\frac{2 \alpha ^2-\beta  (\beta +4)+10}{4 (\alpha -\beta +5)},-\frac{\sqrt{3} \left(\beta ^2-12\right)}{4 (\alpha -\beta +5)}\Big).
\end{align*}
Recall that the boundary $\ell$ of the two Voronoi regions is the perpendicular bisector of the line segment joining $P$ and $Q$, and hence, we have the two canonical equations as $\rho(\tilde d, \tilde p)-\rho(\tilde d, \tilde q)=0$, and $\rho(\tilde e, \tilde p)-\rho(\tilde e, \tilde q)=0$. There is no solution for these two equations yielding the fact that this case cannot happen.

Case~2. $\ell$ intersects the sides $OA_1$ and $A_2A_3$.

 Let $\ell$ intersect $OA_1$ and $A_2A_3$ at the points $D$ and $E$, respectively. Let the parametric values for $D$ and $E$ be, respectively, $\ga$ and $\gb$, i.e., $\tilde d=(\ga, 0)$, and $\tilde e=(\frac{3 (1-\beta )}{2}+\beta ,\frac{1}{2} \sqrt{3} (1-\beta )+\sqrt{3} \beta)$, where $0\leq \ga\leq 1$ and $0\leq \gb\leq 1$. Then, proceeding in the similar way as Case~1, we have
\begin{align*}
\tilde p&=E(X : X\in \gw{DE})=\frac{\int_{\gw{DE}}(x_1, x_2) dP}{\int_{\gw{DE}}  dP}=\Big(\frac{-2 \alpha ^2-\beta ^2+6 \beta +7}{-4 \alpha +4 \beta +8},-\frac{\sqrt{3} (\beta +1)^2}{4 (\alpha -\beta -2)}\Big),
\end{align*}
and
\begin{align*}
\tilde q&=E(X : X \in \gw{ED})=\Big(\frac{2 \alpha ^2+\beta ^2-6 \beta +5}{4 \alpha -4 \beta +16},-\frac{\sqrt{3} \left(\beta ^2+2 \beta -11\right)}{4 (\alpha -\beta +4)}\Big).
\end{align*}
Solving the two canonical equations $\rho(\tilde d, \tilde p)-\rho(\tilde d, \tilde q)=0$, and $\rho(\tilde e, \tilde p)-\rho(\tilde e, \tilde q)=0$, we have $\ga=0$, and $\gb=1$ yielding
$\tilde p=(1,\frac{1}{\sqrt{3}}), \te{ and } \tilde q=(0,\frac{2}{\sqrt{3}})$. Due to symmetry, we can obtain the distortion error due to the points $\tilde p$ and $\tilde q$ as
\begin{align*} \int_{L}\min_{a\in \set{\tilde p, \tilde q}}&\|x-a\|^2 dP=\frac{2}{6} \Big(\int_0^1 \rho((t,0),(1,\frac{1}{\sqrt{3}})) \, dt+  \int_0^{1} \rho((\frac{t}{2}+1,\frac{\sqrt{3} t}{2}),(1,\frac{1}{\sqrt{3}})) \, dt\\
&+ \int_0^1 \rho((\frac{3 (1-t)}{2}+t,\frac{1}{2} \sqrt{3} (1-t)+\sqrt{3} t), (1,\frac{1}{\sqrt{3}})) \, dt
\Big)=\frac 12.
\end{align*}

Case~3. $\ell$ intersects the sides $OA_1$ and $A_3A_4$.

 Let $\ell$ intersect $OA_1$ and $A_3A_4$ at the points $D$ and $E$, respectively. Let the parametric values for $D$ and $E$ be, respectively, $\ga$ and $\gb$, i.e., $\tilde d=(\ga, 0)$, and $\tilde e=(\gb, \sqrt 3)$, where $0\leq \ga\leq 1$ and $0\leq \gb\leq 1$. Then, proceeding in the similar way as Case~1, we have
\begin{align*}
\tilde p&=E(X : X\in \gw{DE})=\frac{\int_{\gw{DE}}(x_1, x_2) dP}{\int_{\gw{DE}}  dP}=\Big(\frac{\alpha ^2+\beta ^2-2 \beta -6}{2 (\alpha -\beta -3)},-\frac{\sqrt{3} (\beta +1)}{\alpha -\beta -3}\Big),
\end{align*}
and
\begin{align*}
\tilde q&=E(X : X \in \gw{ED})=\Big(\frac{\alpha ^2+\beta ^2-2 \beta }{2 \alpha -2 \beta +6},-\frac{\sqrt{3} (\beta -2)}{\alpha -\beta +3}\Big).
\end{align*}
Solving the two canonical equations $\rho(\tilde d, \tilde p)-\rho(\tilde d, \tilde q)=0$, and $\rho(\tilde e, \tilde p)-\rho(\tilde e, \tilde q)=0$, we have the following three sets of solutions:
\[\{\alpha \to 0,\beta \to 0\},  \{\alpha \to \frac{1}{2},\beta \to \frac{1}{2} \},\{\alpha \to 1,\beta \to 1\}.\]
If $\{\alpha \to 0,\beta \to 0\}$, then $\tilde p=(1,\frac{1}{\sqrt{3}})$ and $\tilde q=(0,\frac{2}{\sqrt{3}})$ which falls in Case~2. If $\{\alpha \to \frac{1}{2},\beta \to \frac{1}{2}\}$, then $\tilde p=(\frac{13}{12},\frac{\sqrt{3}}{2})$, and $\tilde q=(-\frac{1}{12},\frac{\sqrt{3}}{2})$, and the corresponding distortion is obtained as
\begin{align*} &\int_{L}\min_{a\in \set{\tilde p, \tilde q}}\|x-a\|^2 dP=\frac{2}{6} \Big(\int_{\frac 12}^1 \rho((t,0),(\frac{13}{12}, \frac{\sqrt{3}}{2})) \, dt+  \int_0^1 \rho((\frac{t}{2}+1,\frac{\sqrt{3} t}{2}),(\frac{13}{12},\frac{\sqrt{3}}{2})) \, dt\\
&+\int_0^1 \rho((\frac{3 (1-t)}{2}+t,\frac{1}{2} \sqrt{3} (1-t)+\sqrt{3} t), (\frac{13}{12},\frac{\sqrt{3}}{2})) \, dt\\
&+\int_0^{\frac 12} \rho((1-t,\sqrt{3} (1-t)+\sqrt{3} t),(\frac{13}{12},\frac{\sqrt{3}}{2})) \, dt
\Big)=\frac{71}{144}=0.493056,
\end{align*}
which is smaller than $\frac 12$, the distortion error obtained in Case~2.
If $\{\alpha \to 1,\beta \to 1\}$, then $\tilde p=(1, \frac{2}{\sqrt{3}})$ and $\tilde q=(0,\frac{1}{\sqrt{3}})$ which falls in Case~2 after giving a rotation of the hexagon with respect to its center by an angle of $\frac \pi 3$. Thus, the distortion error in this case is $\frac 12$ which is same as in Case~2.

Case~4. $\ell$ intersects the sides $OA_1$ and $A_4A_5$.

This case is the reflection of Case~2 with respect to the diagonal $A_1A_4$, and so the same distortion error $\frac 12$ occurs in this case.

Case~5. $\ell$ intersects the sides $OA_1$ and $OA_5$.

This case is the reflection of Case~1 with respect to the line $x_1=\frac 12$, and so, this case cannot happen.

Taking into account all the distortion errors, we see that the distortion error is minimum when  $\tilde p=(\frac{13}{12},\frac{\sqrt{3}}{2})$, and $\tilde q=(-\frac{1}{12},\frac{\sqrt{3}}{2})$. Hence, the optimal set of two-means is $\set{(\frac{13}{12},\frac{\sqrt{3}}{2}), (-\frac{1}{12},\frac{\sqrt{3}}{2})}$, and the quantization error for two-means is $V_2=\frac{71}{144}=0.493056$. Thus, the proof of the proposition is deduced (see Figure~\ref{Fig2}).
\end{proof}

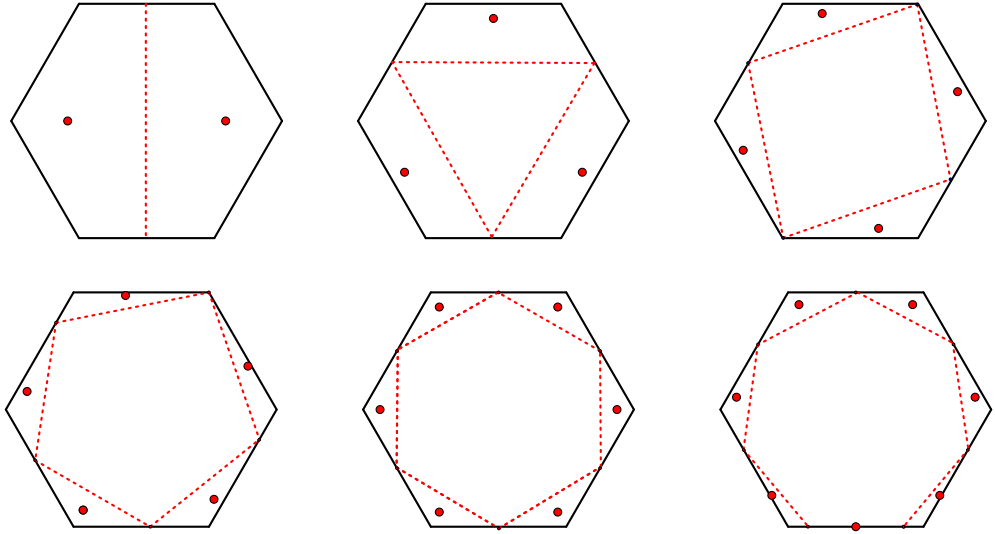
\begin{figure}
\begin{tikzpicture}[line cap=round,line join=round,>=triangle 45,x=0.3cm,y=0.3cm]
\clip(-4.889585772599514,-0.741403825461075) rectangle (10.479614030765603,11.945429448884541);
\draw [line width=0.8pt] (0.,0.)-- (6.,0.);
\draw [line width=0.8pt] (6.,0.)-- (9.,5.196152422706632);
\draw [line width=0.8pt] (9.,5.196152422706632)-- (6.,10.392304845413264);
\draw [line width=0.8pt] (6.,10.392304845413264)-- (0.,10.392304845413264);
\draw [line width=0.8pt] (0.,10.392304845413264)-- (-3.,5.196152422706632);
\draw[line width=0.8pt] (-3.,5.196152422706632)-- (0.,0.);
\draw [line width=0.8 pt,dotted,color=ffqqqq] (2.9635249845395184,10.391881720268692)-- (2.9635249845395184,-0.04563706046323258);
\begin{scriptsize}
\draw [fill=ffqqqq] (6.5,5.196152422706632) circle (1.5pt);
\draw [fill=ffqqqq] (-0.5,5.196152422706632) circle (1.5pt);
\end{scriptsize}
\end{tikzpicture}\begin{tikzpicture}[line cap=round,line join=round,>=triangle 45,x=0.3cm,y=0.3cm]
\clip(-4.889585772599514,-0.741403825461075) rectangle (10.479614030765603,11.945429448884541);
\draw [line width=0.8pt](0.,0.)-- (6.,0.);
\draw [line width=0.8pt](6.,0.)-- (9.,5.196152422706632);
\draw [line width=0.8pt](9.,5.196152422706632)-- (6.,10.392304845413264);
\draw [line width=0.8pt](6.,10.392304845413264)-- (0.,10.392304845413264);
\draw [line width=0.8pt] (0.,10.392304845413264)-- (-3.,5.196152422706632);
\draw [line width=0.8pt] (-3.,5.196152422706632)-- (0.,0.);
\draw [line width=0.8pt,dotted,color=ffqqqq] (-1.4911292573502182,7.808103486245999)-- (7.477259521414533,7.76477793659013);
\draw [line width=0.8pt,dotted,color=ffqqqq] (7.477259521414533,7.76477793659013)-- (2.9280768075483548,0.05283009784555518);
\draw [line width=0.8pt,dotted,color=ffqqqq] (2.9280768075483548,0.05283009784555518)-- (-1.4911292573502182,7.808103486245999);
\begin{scriptsize}
\draw [fill=ffqqqq] (6.9375,2.92283573777248) circle (1.5pt);
\draw [fill=ffqqqq] (3.,9.742785792574935) circle (1.5pt);
\draw [fill=ffqqqq] (-0.9375,2.92283573777248) circle (1.5pt);
\end{scriptsize}
\end{tikzpicture}
\begin{tikzpicture}[line cap=round,line join=round,>=triangle 45,x=0.3cm,y=0.3cm]
\clip(-4.889585772599514,-0.741403825461075) rectangle (10.479614030765603,11.945429448884541);
\draw [line width=0.8pt](0.,0.)-- (6.,0.);
\draw [line width=0.8pt](6.,0.)-- (9.,5.196152422706632);
\draw [line width=0.8pt](9.,5.196152422706632)-- (6.,10.392304845413264);
\draw [line width=0.8pt](6.,10.392304845413264)-- (0.,10.392304845413264);
\draw [line width=0.8pt](0.,10.392304845413264)-- (-3.,5.196152422706632);
\draw[line width=0.8pt] (-3.,5.196152422706632)-- (0.,0.);
\draw [line width=0.8pt,dotted,color=ffqqqq] (-1.5344548070060866,7.76477793659013)-- (5.960865283459141,10.364310915942234);
\draw [line width=0.8pt,dotted,color=ffqqqq] (5.960865283459141,10.364310915942234)-- (7.433933971758665,2.6090375275417905);
\draw [line width=0.8pt,dotted,color=ffqqqq] (-1.5344548070060866,7.76477793659013)-- (0.0252649806051745,0.009504548189686782);
\draw [line width=0.8pt,dotted,color=ffqqqq] (0.0252649806051745,0.009504548189686782)-- (7.433933971758665,2.6090375275417905);
\begin{scriptsize}
\draw [fill=ffqqqq] (4.25,0.4330127018922193) circle (1.5pt);
\draw [fill=ffqqqq] (7.75,6.495190528383289) circle (1.5pt);
\draw [fill=ffqqqq] (1.75,9.959292143521044) circle (1.5pt);
\draw [fill=ffqqqq] (-1.75,3.8971143170299736) circle (1.5pt);
\draw [fill=qqqqff] (-1.5344548070060866,7.76477793659013) circle (0.5pt);
\draw [fill=qqqqff] (5.960865283459141,10.364310915942234) circle (0.5pt);
\draw [fill=qqqqff] (7.433933971758665,2.6090375275417905) circle (0.5pt);
\draw [fill=qqqqff] (0.0252649806051745,0.009504548189686782) circle (0.5pt);
\end{scriptsize}
\end{tikzpicture}
\begin{tikzpicture}[line cap=round,line join=round,>=triangle 45,x=0.3cm,y=0.3cm]
\clip(-4.889585772599514,-0.741403825461075) rectangle (10.479614030765603,11.945429448884541);
\draw [line width=0.8pt](0.,0.)-- (6.,0.);
\draw [line width=0.8pt](6.,0.)-- (9.,5.196152422706632);
\draw[line width=0.8pt] (9.,5.196152422706632)-- (6.,10.392304845413264);
\draw [line width=0.8pt](6.,10.392304845413264)-- (0.,10.392304845413264);
\draw [line width=0.8pt](0.,10.392304845413264)-- (-3.,5.196152422706632);
\draw [line width=0.8pt](-3.,5.196152422706632)-- (0.,0.);
\draw [line width=0.8pt,dotted,color=ffqqqq] (-0.7746103879911016,9.050640297342035)-- (-1.7000891347211804,2.9446407587328944);
\draw [line width=0.8pt,dotted,color=ffqqqq] (-1.7000891347211804,2.9446407587328944)-- (3.396780481098202,0.);
\draw [line width=0.8pt,dotted,color=ffqqqq] (3.396780481098202,0.)-- (8.236234235461756,3.8732713134447065);
\draw [line width=0.8pt,dotted,color=ffqqqq] (-0.7746103879911019,9.050640297342037)-- (6.,10.392304845413264);
\draw [line width=0.8pt,dotted,color=ffqqqq] (6.,10.392304845413264)-- (8.236234235461758,3.873271313444708);
\begin{scriptsize}
\draw [fill=ffqqqq] (0.4257,0.737334) circle (1.5pt);
\draw [fill=ffqqqq] (6.22419,1.21614) circle (1.5pt);
\draw [fill=ffqqqq] (7.72859,7.12179) circle (1.5pt);
\draw [fill=ffqqqq] (0.4257,0.737334) circle (1.5pt);
\draw [fill=ffqqqq] (2.30335,10.254) circle (1.5pt);
\draw [fill=ffqqqq] (-2.05889,5.99838) circle (1.5pt);
\draw [fill=ffqqqq] (-0.776469,9.04742) circle (0.5pt);
\draw [fill=ffqqqq] (-1.7028,2.94934) circle (0.5pt);
\draw [fill=ffqqqq] (3.4056,0.) circle (0.5pt);
\draw [fill=ffqqqq] (8.22353,3.85127) circle (0.5pt);
\draw [fill=ffqqqq] (6.,10.392304845413264) circle (0.5pt);
\end{scriptsize}
\end{tikzpicture}
\begin{tikzpicture}[line cap=round,line join=round,>=triangle 45,x=0.3cm,y=0.3cm]
\clip(-4.889585772599514,-0.741403825461075) rectangle (10.479614030765603,11.945429448884541);
\draw [line width=0.8pt](0.,0.)-- (6.,0.);
\draw[line width=0.8pt] (6.,0.)-- (9.,5.196152422706632);
\draw[line width=0.8pt] (9.,5.196152422706632)-- (6.,10.392304845413264);
\draw [line width=0.8pt](6.,10.392304845413264)-- (0.,10.392304845413264);
\draw [line width=0.8pt](0.,10.392304845413264)-- (-3.,5.196152422706632);
\draw [line width=0.8pt](-3.,5.196152422706632)-- (0.,0.);
\draw [line width=0.8pt,dotted,color=ffqqqq] (2.971021659671325,10.392304845413264)-- (-1.464169149001539,7.856289488467713);
\draw [line width=0.8pt,dotted,color=ffqqqq] (-1.464169149001539,7.856289488467713)-- (-1.517344191517739,2.628117232278245);
\draw [line width=0.8pt,dotted,color=ffqqqq] (-1.517344191517739,2.628117232278245)-- (3.0162266373492934,-0.07663369307186471);
\draw [line width=0.8pt,dotted,color=ffqqqq] (3.0162266373492934,-0.07663369307186471)-- (7.533730595296208,2.6564993161758923);
\draw [line width=0.8pt,dotted,color=ffqqqq] (2.971021659671325,10.392304845413264)-- (-1.464169149001539,7.856289488467713);
\draw [line width=0.8 pt,dotted,color=ffqqqq] (-1.464169149001539,7.856289488467713)-- (-1.517344191517739,2.628117232278245);
\draw [line width=0.8 pt,dotted,color=ffqqqq] (-1.517344191517739,2.628117232278245)-- (3.0162266373492934,-0.07663369307186471);
\draw [line width=0.8 pt,dotted,color=ffqqqq] (3.0162266373492934,-0.07663369307186471)-- (7.533730595296208,2.6564993161758923);
\draw [line width=0.8 pt,dotted,color=ffqqqq] (7.533730595296208,2.6564993161758923)-- (7.5,7.794228634059947);
\draw [line width=0.8 pt,dotted,color=ffqqqq] (7.5,7.794228634059947)-- (2.971021659671325,10.392304845413264);
\begin{scriptsize}
\draw [fill=ffqqqq] (0.375,0.649519052838329) circle (1.5pt);
\draw [fill=ffqqqq] (5.625,0.649519052838329) circle (1.5pt);
\draw [fill=ffqqqq] (8.25,5.196152422706632) circle (1.5pt);
\draw [fill=ffqqqq] (5.625,9.742785792574935) circle (1.5pt);
\draw [fill=ffqqqq] (0.375,9.742785792574935) circle (1.5pt);
\draw [fill=ffqqqq] (-2.25,5.196152422706632) circle (1.5pt);
\draw [fill=ffqqqq] (-1.5,2.598076211353316) circle (0.5pt);
\draw [fill=ffqqqq] (-1.5,7.794228634059947) circle (0.5pt);
\draw [fill=ffqqqq] (3.,10.392304845413264) circle (0.5pt);
\draw [fill=ffqqqq] (7.5,7.794228634059947) circle (0.5pt);
\draw [fill=ffqqqq] (7.5,2.598076211353316) circle (0.5pt);
\draw [fill=ffqqqq] (3.0162266373492934,-0.07663369307186471) circle (0.5pt);
\draw [fill=qqqqff] (-21.575281219465545,17.869742445081663) circle (2.5pt);
\end{scriptsize}
\end{tikzpicture}
\begin{tikzpicture}[line cap=round,line join=round,>=triangle 45,x=0.3cm,y=0.3cm]
\clip(-4.889585772599514,-0.741403825461075) rectangle (10.479614030765603,11.945429448884541);
\draw [line width=0.8pt] (0.,0.)-- (6.,0.);
\draw [line width=0.8pt] (6.,0.)-- (9.,5.196152422706632);
\draw [line width=0.8pt] (9.,5.196152422706632)-- (6.,10.392304845413264);
\draw [line width=0.8pt](6.,10.392304845413264)-- (0.,10.392304845413264);
\draw [line width=0.8pt](0.,10.392304845413264)-- (-3.,5.196152422706632);
\draw[line width=0.8pt] (-3.,5.196152422706632)-- (0.,0.);
\draw [line width=0.8 pt,dotted,color=ffqqqq] (3.078195511176974,10.392304845413264)-- (-1.3395094073730982,8.072206494626581);
\draw [line width=0.8 pt,dotted,color=ffqqqq] (-1.3395094073730982,8.072206494626581)-- (-1.9768819777718163,3.424060026068034);
\draw [line width=0.8 pt,dotted,color=ffqqqq] (-1.9768819777718163,3.424060026068034)-- (0.8972130299975904,0.);
\draw [line width=0.8 pt,dotted,color=ffqqqq] (3.078195511176974,10.392304845413264)-- (7.293309976190826,8.152226256715059);
\draw [line width=0.8 pt,dotted,color=ffqqqq] (7.293309976190826,8.152226256715059)-- (7.9859213097111965,3.4397166082535198);
\draw [line width=0.8 pt,dotted,color=ffqqqq] (7.9859213097111965,3.4397166082535198)-- (5.116940004453355,0.);
\begin{scriptsize}
\draw [fill=ffqqqq] (3.,0.) circle (1.5pt);
\draw [fill=ffqqqq] (6.72171,1.38947) circle (1.5pt);
\draw [fill=ffqqqq] (8.28689,5.74371) circle (1.5pt);
\draw [fill=ffqqqq] (5.51846,9.85003) circle (1.5pt);
\draw [fill=ffqqqq] (0.481541,9.85003) circle (1.5pt);
\draw [fill=ffqqqq] (-2.28689,5.74371) circle (1.5pt);
\draw [fill=ffqqqq] (-0.721713,1.38947) circle (1.5pt);
\draw [fill=ffqqqq] (-1.33154,8.08601) circle (0.5pt);
\draw [fill=ffqqqq] (-1.96381,3.40142) circle (0.5pt);
\draw [fill=ffqqqq] (0.879769,0.) circle (0.5pt);
\draw [fill=ffqqqq] (5.12023,0.) circle (0.5pt);
\draw [fill=ffqqqq] (7.96381,3.40142) circle (0.5pt);
\draw [fill=ffqqqq] (7.33154,8.08601) circle (0.5pt);
\draw [fill=ffqqqq] (3.,10.392304845413264) circle (0.5pt);
\draw [fill=ffqqqq] (5.116940004453354,0.) circle (0.5pt);
\end{scriptsize}
\end{tikzpicture}
\caption{Points in an optimal set of $n$-means for $2\leq n\leq 7$.} \label{Fig2}
\end{figure}

\begin{prop} \label{prop02}
The set $\set{(\frac{37}{32},\frac{9 \sqrt{3}}{32}), (\frac{1}{2},\frac{15 \sqrt{3}}{16}), (-\frac{5}{32},\frac{9 \sqrt{3}}{32})}$ forms an optimal set of three-means with quantization error $V_3=\frac{199}{768}=0.259115$.
\end{prop}

 \begin{proof}
 The proof of the proposition can be deduced by considering different cases as it was done in Proposition~\ref{prop01}. To avoid too much technicality, we will prove it in a different way. Recall that the probability distribution is uniform on the boundary of the regular hexagon, and so we can assume that the Voronoi regions of the elements in an optimal set of three-means will partition the boundary of the hexagon into three equal parts. Let the points $P$, $Q$, and $R$ form an optimal set of three-means. Let the boundaries of the Voronoi regions cut the sides $OA_1$, $A_2A_3$, and $A_4A_5$ at the points $D$, $E$, and $F$ with parameters, respectively, given by $\ga$, $\gb$, and $\gg$. Let $P$, $Q$, and $R$ lie in the Voronoi regions that contain the vertices $A_1$, $A_3$, and $A_5$, respectively. Then, $\tilde p=E(X : X\in \gw{DE})$, $\tilde q=E(X : X \in \gw{EF})$, and $\tilde r=E(X : X \in \gw{FD})$. Then, proceeding as Case~1 of Proposition~\ref{prop01}, we have
 \begin{align*}
 \tilde p&=\Big(\frac{-2 \alpha ^2-\beta ^2+6 \beta +7}{-4 \alpha +4 \beta +8},-\frac{\sqrt{3} (\beta +1)^2}{4 (\alpha -\beta -2)}\Big),\\
 \tilde q&=\Big(\frac{\beta ^2-6 \beta -\gamma ^2+7}{-4 \beta +4 \gamma +8},\frac{\sqrt{3} \left(\beta ^2+2 \beta +\gamma ^2-4 \gamma -7\right)}{4 (\beta -\gamma -2)}\Big),\\
 \tilde r&=\Big(\frac{2 \alpha ^2+\gamma ^2-2}{4 (\alpha -\gamma +2)},\frac{\sqrt{3} (\gamma -2)^2}{4 (\alpha -\gamma +2)}\Big).
 \end{align*}
 Solving the canonical equations $\rho(\tilde d, \tilde r)-\rho(\tilde d, \tilde p)=0$, $\rho(\tilde e, \tilde p)-\rho(\tilde e, \tilde q)=0$, and $\rho(\tilde f, \tilde q)-\rho(\tilde f, \tilde r)=0$, we have three sets of solutions:
 \[\{\alpha \to 0,\beta \to 0,\gamma \to 0\}, \{\alpha \to \frac{1}{2},\beta \to \frac{1}{2}, \gamma \to \frac{1}{2}\}, \te{ and } \{\alpha \to 1,\beta \to 1,\gamma \to 1\}.\]
 If $\{\alpha \to 0,\beta \to 0,\gamma \to 0\}$, then $\tl p=(\frac{7}{8},\frac{\sqrt{3}}{8})$, $\tl q=(\frac{7}{8},\frac{7 \sqrt{3}}{8})$, and $\tl r=(-\frac{1}{4},\frac{\sqrt{3}}{2})$ yielding the distortion error
 \begin{align*} &\int_{L}\min_{a\in \set{\tilde p, \tilde q, \tl r}}\|x-a\|^2 dP=\frac{3}{6} \Big(\int_0^1 \rho((\tl t, 0), \tl p) \, dt+  \int_0^1 \rho((\frac{t}{2}+1,\frac{\sqrt{3} t}{2}), \tl p) \, dt\Big)=\frac{13}{48}.
\end{align*}
Similarly, if $\{\alpha \to \frac{1}{2},\beta \to \frac{1}{2}, \gamma \to \frac{1}{2}\}$, then $\tl p=(\frac{37}{32},\frac{9 \sqrt{3}}{32})$, $\tl q=(\frac{1}{2},\frac{15 \sqrt{3}}{16})$, and $\tl r=(-\frac{5}{32},\frac{9 \sqrt{3}}{32})$ yielding the distortion error $\frac{199}{768}$; and if $\{\alpha \to 1,\beta \to 1,\gamma \to 1\}$, then $\tl p=(\frac{5}{4},\frac{\sqrt{3}}{2})$, $\tl q=(\frac{1}{8},\frac{7 \sqrt{3}}{8})$, and $\tl r=(\frac{1}{8},\frac{\sqrt{3}}{8})$ with distortion error $\frac{13}{48}$. Since among the distortion errors, $\frac{199}{768}=0.259115<0.270833=\frac{13}{48}$, we can say that the set $\set{(\frac{37}{32},\frac{9 \sqrt{3}}{32}), (\frac{1}{2},\frac{15 \sqrt{3}}{16}), (-\frac{5}{32},\frac{9 \sqrt{3}}{32})}$ forms an optimal set of three-means with quantization error $V_3=\frac{199}{768}=0.259115$ (see Figure~\ref{Fig2}).
 \end{proof}

Proceeding in the similar way as Proposition~\ref{prop02}, the following two propositions can be proved.
 \begin{prop} \label{prop03}
 The set $\set{(\frac{17}{24},\frac{1}{8 \sqrt{3}}), (\frac{31}{24},\frac{5 \sqrt{3}}{8}), (\frac{7}{24},\frac{23}{8 \sqrt{3}}), (-\frac{7}{24},\frac{3 \sqrt{3}}{8})}$ forms an optimal set of four-means with quantization error $V_4=\frac{23}{144}=0.159722$ (see Figure~\ref{Fig2}) .
 \end{prop}
\begin{prop} \label{prop04}
 The set $\set{(0.07095, 0.122889), (1.03737, 0.20269), (1.2881, 1.18696), \\(0.383892, 1.70901), (-0.343148, 0.99973)} $
 forms the optimal set of five-means with quantization error $V_5=0.10509 $(see Figure~\ref{Fig2}).
 \end{prop}

%%\begin{prop} \label{prop06}
%% The set $
%% \set{(\frac{1}{2},0), (1.12029,0.231579),  (1.38115,0.957286), (0.919743,1.64167), \\
%%  (0.0802568, 1.64167), (-0.381148, 0.957286),
%%  (-0.120285, 0.231579)}$
%% forms an optimal set of seven-means with quantization error $V_7=0.0541846$ (see Figure~\ref{Fig2}).
%% \end{prop}

The following theorem determines the optimal sets of $n$-means and the $n$th quantization errors for all positive integers $n$ of the form $n=6k$, where $k\in \D N$. It also helps us to determine the quantization dimension and the quantization coefficient for the uniform distribution defined on the boundary of the regular hexagon.

\begin{theorem} \label{Th41} Let $n\in \D N$ be such that $n=6k$ for some positive integer $k$. Then, the optimal set of $n$-means for $P$ is given by
\begin{align*} \ga_n&=\Big\{(\frac{r }{8},\frac{\sqrt{3} r }{8}),(1-\frac{r }{8},\frac{\sqrt{3} r }{8}),(\frac{6-r }{4},\frac{\sqrt{3}}{2}),(1-\frac{r }{8},-\frac{1}{8} \sqrt{3} (r -8)),\\
&(\frac{r }{8},-\frac{1}{8} \sqrt{3} (r -8)),(\frac{r -2}{4},\frac{\sqrt{3}}{2}) \Big\} \uu \gg\uu \mathop{\UU}\limits_{i=1}^5T_i(\gg),
\end{align*}
where $\gg:=\left\{\begin{array}{cc}
\set{r+\frac{2j-1}{2(k-1)}(1-2r) : j=1, 2, \cdots, (k-1)} &\te{ if }k\geq 2 \\
\es & k=1
\end{array}
\right.$,
and $r=\frac{2 (\sqrt{13} (k-1)-4)}{13 (k-1)^2-16}$, and  $T_i$ for $1\leq i\leq 5$, are five affine transformations on $\D R^2$, such that $T_i(OA_1)=A_iA_{i+1}$ for $1\leq i\leq 4$,  and $T_5(OA_1)=A_5O$.
The quantization error for $n$-means is given by
\[V_n=\frac{13 \left(13 (k-1)^2-8 \sqrt{13} (k-1)+16\right)}{12 \left(16-13 (k-1)^2\right)^2}.\]
\end{theorem}

\begin{proof} Let us first prove the theorem for $n=6$, i.e., when $k=1$.
Recall that the probability distribution is uniform on the boundary of the regular hexagon, and so we can assume that the Voronoi regions of the elements in an optimal set of six-means will partition the boundary of the hexagon into six equal parts. Let the boundaries of the Voronoi regions cut the sides $OA_1, A_1A_2, A_2A_3, A_3A_4, A_4A_5$, and $A_5O$ of the hexagon given in Figure~\ref{Fig1} at the points $D_1, D_2, \cdots, D_6$, respectively. Then, we must have $|OD_1|=|A_1D_2|=|A_2D_3|=\cdots=|A_5D_6|$ which equals $r$, say, where $0\leq r \leq 1$. Let $\set{\tilde p_1, \tilde p_2, \cdots, \tilde p_6}$ forms an optimal set of six-means, where
 $\tilde p_1=E(X : X\in \gw{D_1D_2}), \, \tilde p_2=E(X : X\in \gw{D_2D_3}), \cdots, \tilde p_5=E(X : X\in \gw{D_5D_6}), \te{ and } \tilde p_6=E(X : X\in \gw{D_6D_1})$.
  Then, we have
  \[\tilde p_1=\frac{\int_r^1 (t, 0) \, dt+\int_0^r (\frac{t}{2}+1, \frac{\sqrt{3} t}{2}) \, dt}{\int_r^1 1 \, dt+\int_0^r 1 \, dt}=(-\frac{r^2}{4}+r+\frac{1}{2},\frac{\sqrt{3} r^2}{4}).\]
  Similarly, $\tilde p_2=(\frac{1}{4} \left(-2 r^2+2 r+5\right),\frac{1}{4} \sqrt{3} (2 r+1))$, and so on.
Solving the canonical equations $\rho(\tilde d_2, \tilde p_1)-\rho(\tilde d_2, \tilde p_2)=0$, we have $r=0, \frac 12, 1$. If $r=0$, then we see that $D_1, D_2\cdots, D_6$ coincide with $O, A_1, A_2, \cdots, A_5$, respectively. On the other hand, if $r=1$, then $D_1, D_2\cdots, D_6$ coincide with $A_1, A_2, \cdots, A_5, O$, respectively. Thus, if $r=0$, or if $r=1$, then the optimal set of six-means consists of the midpoints of the sides of the hexagon, and hence, the distortion error for six-means in these two cases is given by
\begin{equation*} 6\int_0^1 \rho((t, 0), (\frac{1}{2},0))\, dP=\int_0^1 \rho((t, 0), (\frac{1}{2},0)) \,dt=\frac 1 {12}.
\end{equation*}
If $r=\frac 12$, then $D_1, D_2\cdots, D_6$ coincide with the midpoints of the sides of $OA_1, A_1A_2, \cdots, A_5O$, respectively, yielding
$\tilde p_1=(\frac{15}{16},\frac{\sqrt{3}}{16}), \tilde p_2=(\frac{11}{8},\frac{\sqrt{3}}{2}), \tilde p_3=(\frac{15}{16},\frac{15 \sqrt{3}}{16}), \tilde p_4= (\frac{1}{16},\frac{15 \sqrt{3}}{16}), \tilde p_5=(-\frac{3}{8},\frac{\sqrt{3}}{2})$, and $\tilde p_6=(\frac{1}{16},\frac{\sqrt{3}}{16})$ (see Figure~\ref{Fig2}), and the corresponding distortion error is given by
  \begin{equation*} 6\Big(\int_\frac 12 ^1 \rho((t, 0), \tilde p_1)\, dP+\int_0^{\frac 12} \rho((\frac{t}{2}+1, \frac{\sqrt{3} t}{2}), \tilde p_1), dP\Big)=\frac{13}{192}.
\end{equation*}
Since $\frac{13}{192}<\frac 1 {12}$, the set $\set{\tilde p_1, \tilde p_2, \cdots, \tilde p_6}$ obtained for $r=\frac 12$, gives the optimal set of six-means with quantization error  $V_6=\frac{13}{192}$.
Notice that for $k=1$, we have $r=\frac 12$, and $\gg=\es$. Thus, the statement of the proposition is true for $n=6$, i.e., when $k=1$. Let $n=6k$ for $k\geq 2$. Then, as the hexagon is a regular polygon, and the probability distribution $P$ is uniform, it is not difficult to show that the optimal set $\ga_n$ contains six elements from each of the interior angles,  $k-1$ elements from each side of the hexagon.
Let $a, b, c, d, e, f$ be the six points that $\ga_n$ contains from the interior of the angles $\angle O, \angle A_1, \angle A_2, \angle A_3, \angle A_4, \angle A_5$, respectively. As the optimal set $\ga_n$ contains $k-1$ elements from each side of the hexagon, and $P$ is a uniform probability distribution, the Voronoi regions of $a, b, c, d, e, f$ will form isosceles triangles $P$-almost surely. Let the length of each of the two equal sides of the isosceles triangle formed by the Voronoi regions equal $r$. Then,
\[a=\frac{\int_{1-r}^1 (\frac{t-1}{2},\frac{1}{2} \sqrt{3} (1-t)) \, dt+\int_0^r (t,0)\, dt}{\int_{1-r}^1 1 \, dt+\int_0^r 1 \, dt}=(\frac{r}{8},\frac{\sqrt{3} r}{8}).\]
Similarly,
\begin{align*} b=(1-\frac{r }{8},\quad &\frac{\sqrt{3} r }{8}), \quad   c=(\frac{6-r }{4},\frac{\sqrt{3}}{2}), \quad d=(1-\frac{r }{8},-\frac{1}{8} \sqrt{3} (r -8)), \\
&e=(\frac{r }{8},-\frac{1}{8} \sqrt{3} (r -8)), \te{ and } f=(\frac{r -2}{4},\frac{\sqrt{3}}{2}).
 \end{align*}Let $\gg$ be the set of all the $k-1$ points that $\ga_n$ contains from the side $OA$. Then, the Voronoi regions of the points in $\gg$ covers the closed interval $[r, 1-r]$ yielding
\[\gg=\Big\{\Big(r+\frac{2j-1}{2(k-1)}(1-2r), 0\Big) : j=1, 2, \cdots, (k-1)\Big\}.\]
Let $T_i$ for $1\leq i\leq 5$ be the affine transformations on $\D R^2$ as given in the hypothesis. Then, the set of points that $\ga_n$ contains from the sides $A_iA_j$ is $T_i(\gg)$, and the set of points that $\ga_n$ contains from the side $A_5O$ is $T_5(\gg)$, yielding
\begin{align*} \ga_n=\{a, b, c, d, e, f\}\uu \gg\uu \mathop{\UU}\limits_{i=1}^5T_i(\gg).
\end{align*}
Using the symmetry, the quantization error for $n$-means is obtained as
\begin{align*}
V_n&=6(\te{quantization error due to the points  $a$ and the $(k-1)$ points in $\gg$})\\
&=\frac 6 6\Big(\int_{1-r}^1 \rho((\frac{t-1}{2},\frac{1}{2} \sqrt{3} (1-t)),a) \, dt+\int_0^r \rho((t, 0),a) \, dt\\
&\qquad \qquad +(k-1) \Big(\int_{r }^{r +\frac{1-2 r }{(k-1)}} \rho\Big((t, 0), (r+\frac{1}{2(k-1)}(1-2r), 0)\Big) dt\Big)\\
&=\frac 1{24(k-1)^2}\Big(24 r ^2-12 r +r ^3 \left(13 (k-1)^2-16\right)+2\Big).
\end{align*}
Notice that for a given $k$, the quantization error $V_n$ is a function of $r$. Solving $\frac{\pa V_n}{\pa r}=0$, we have $r=\frac{2 (\sqrt{13} (k-1)-4)}{13 (k-1)^2-16}$. Putting $r=\frac{2 (\sqrt{13} (k-1)-4)}{13 (k-1)^2-16}$, we have
\[V_n=\frac{13 \left(13 (k-1)^2-8 \sqrt{13} (k-1)+16\right)}{12 \left(16-13 (k-1)^2\right)^2}.\]
Thus, the proof of the theorem is complete.
\end{proof}

\begin{remark} By Proposition~\ref{prop01}, we see that the boundary of the Voronoi regions of the optimal set of two-means bisects the two opposite sides of the hexagon. Due to rotational symmetry, there are three different optimal sets of two-means. Proposition~\ref{prop02} implies that the points in an optimal set of three-means form an equilateral triangle. Due to rotational symmetry, there are two different optimal sets of three-means. Proposition~\ref{prop03} yields that the points in an optimal set of four-means form a rectangle of side lengths $\frac { 7} 6$ and $\frac 2{\sqrt 3}$. Due to rotational symmetry, there are three different optimal sets of four-means. Proposition~\ref{prop04} implies that there are five different optimal sets of five-means.
Theorem~\ref{Th41} implies that an optimal set of $n$-means, where $n=6k$ for some $k\in \D N$, is unique. We observe that the points in an optimal set of six-means form a regular hexagon of side length $\frac { 7} 8$. On the other hand, if $n$ is of the form $n=6k+m$, where $1\leq m\leq 5$, then we can show that an optimal set of $n$-means contains six elements from each of the interior angles, $k$ elements from each of $m$ sides, $k-1$ elements from each of the remaining $6-m$ sides, and the number of such sets is $^6C_m$. The proof being too technical, we skip the proof of it in the paper.
\end{remark}

\begin{prop} \label{prop000}
For sufficiently large $n$ the points in an optimal set of $n$-means lie on the boundary of the hexagon.
\end{prop}
\begin{proof} For large $n$ there exists a unique positive integer $\ell(n)$ such that
\begin{equation} \label{eq1} 6\ell(n)\leq n<6(\ell(n)+1).\end{equation}
Let $a_n, b_n, c_n, d_n, e_n, f_n$ be the six points that an optimal set $\ga_n$ contains from the interior, respectively, of the angles $\angle O, \angle A_1, \angle A_2, \angle A_3, \angle A_4, \angle A_5$. By Theorem~\ref{Th41}, we have
\[a_{6\ell(n)}=\left(\frac{\sqrt{13} (\ell(n)-1)-4}{4 \left(13 (\ell(n)-1)^2-16\right)},\frac{\sqrt{3} \left(\sqrt{13} (\ell(n)-1)-4\right)}{4 \left(13 (\ell(n)-1)^2-16\right)}\right) \te{ implying } \lim_{\ell(n)\to \infty} a_{6\ell(n)}=(0, 0).\]
Similarly, $\mathop{\lim}\limits_{\ell(n)\to \infty} a_{6(\ell(n)+1)}=(0, 0)$. Thus, \eqref{eq1} implies that $\mathop{\lim}\limits_{n\to \infty} a_n=(0, 0)=O$. Similarly, we can show that if $n\to \infty$, then $b_n\to A_1$, $c_n\to A_2$, $d_n\to A_3$, $e_n\to A_4$, and $f_n\to A_5$. The rest of the points in $\ga_n$ are already on the boundary of the hexagon. Thus, the proof of the proposition is complete.
\end{proof}

\begin{prop} \label{prop55}
Quantization dimension $D(P)$ of the uniform distribution $P$ defined on the boundary of the regular hexagon equals the dimension of the boundary of the hexagon. Moreover, the quantization coefficient exists as a finite positive number which equals $3$.
\end{prop}
\begin{proof} For $n\in \D N$, $n\geq 6$, let $\ell(n)$ be the unique positive integer such that $6\ell(n)\leq n<6(\ell(n)+1)$. Then,
$V_{6(\ell(n)+1)}<V_n\leq V_{6\ell(n)}$ implying
\begin{align} \label{eq45}
\frac {2 \log (6\ell(n))}{-\log V_{6(\ell(n)+1)}}<\frac {2\log n}{-\log V_n} <\frac{2 \log (6(\ell(n)+1))}{-\log V_{6\ell(n)}}.
\end{align}
Notice that
\[\lim_{n\to \infty} \frac {2 \log (6\ell(n))}{-\log V_{6(\ell(n)+1)}}=\underset{\ell(n)\to \infty }{\text{lim}}\frac{2 \log (6 \ell(n))}{-\log \left(\frac{13 \left(13 \ell(n)^2-8 \sqrt{13} \ell(n)+16\right)}{12 \left(16-13 \ell(n)^2\right)^2}\right)}=1,\]
and \[\lim_{n\to \infty} \frac{2 \log (6(\ell(n)+1))}{-\log V_{6\ell(n)}}=\underset{\ell(n)\to \infty }{\text{lim}}\frac{2 \log (6 (\ell(n)+1))}{-\log \left(\frac{13 \left(13 (\ell(n)-1)^2-8 \sqrt{13} (\ell(n)-1)+16\right)}{12 \left(16-13 (\ell(n)-1)^2\right)^2}\right)}=1\]
and hence, by \eqref{eq45}, $\mathop{\lim}\limits_{n\to \infty} \frac {2\log n}{-\log V_n}=1$ which is the dimension of the underlying space. Again,
\begin{equation} \label{eq46} (6\ell(n))^2V_{6(\ell(n)+1)}<n^2 V_n<(6(\ell(n)+1))^2V_{6\ell(n)}.
\end{equation}
We have
\[\lim_{n \to \infty} (6\ell(n))^2V_{6(\ell(n)+1)}=\underset{\ell(n)\to \infty }{\text{lim}}(6\ell(n))^2\frac{13 \left(13 \ell(n)^2-8 \sqrt{13} \ell(n)+16\right)}{12 \left(16-13 \ell(n)^2\right)^2}=3,\]
and
\[\lim_{n \to \infty} (6(\ell(n)+1))^2V_{6\ell(n)}=\underset{\ell(n)\to \infty }{\text{lim}}(6(\ell(n)+1))^2 \frac{13 \left(13 (\ell(n)-1)^2-8 \sqrt{13} (\ell(n)-1)+16\right)}{12 \left(16-13 (\ell(n)-1)^2\right)^2}=3,\]
and hence, by \eqref{eq46} we have
$\mathop{\lim}\limits_{n\to\infty} n^2 V_n=3$, i.e., the quantization coefficient exists as a finite positive number which equals $3$.
Thus, the proof of the proposition is complete.
\end{proof}

 \begin{remark}
 Proceeding in the similar way, as it is done for the unform distribution on the boundary of the regular hexagon, we can determine the optimal sets of $n$-means and the $n$th quantization error for any positive integer $n$ for the uniform distribution on the boundary of any regular $m$-sided polygon for $m\geq 3$.
 \end{remark}

 \section{Quantization for a mixed uniform distribution on a semicircular curve} \label{sec2}
The basic definitions and notations that were defined in the first paragraph of Section~\ref{sec1} are also true in this section.

We need the following proposition which generalizes a theorem in \cite{RR2}.
\begin{prop} \label{prop32}
Let $\ga_n$ be an optimal set of $n$-means for a uniform distribution on the unit circular arc $S$ given by
\[S:=\set{(\cos \gq, \sin \gq) : \ga\leq \gq\leq \gb},\]
where $0\leq \ga<\gb\leq 2\pi$. Then,
\[\ga_n:=\Big\{ \frac {2n}{\gb-\ga} \sin(\frac{\gb-\ga}{2n}) \Big(\cos \Big(\ga+(2j-1){\frac{\gb-\ga}{2n}}\Big), \   \sin  \Big(\ga+(2j-1){\frac{\gb-\ga}{2n}}\Big)\Big) : j=1, 2, \cdots, n \Big  \} \]
forms an optimal set of $n$-means, and the corresponding quantization error is given by
\[V_n=1-\frac{4n^2}{(\ga-\gb)^2} \sin^2\frac{\ga-\gb}{2n}.\]
\end{prop}
\begin{proof}
Notice that $S$ is an arc of the unit circle $x_1^2+x_2^2=1$ which subtends a central angle of $\gb-\ga$ radian, and the probability distribution is uniform on $S$. Hence, the density function is given by $f(x_1, x_2)=\frac 1{\gb-\ga}$ if $(x_1, x_2)\in S$, and zero, otherwise. Thus, the proof follows in the similar way as a similar theorem in \cite{RR2}.
\end{proof}

Let $L$ be the boundary of the semicircular disc $x_1^2+x_2^2=1$, where $x_2\geq 0$. Let the base of the semicircular disc be $AOB$, where $A$ and $B$ have the coordinates $(-1, 0)$ and $(1,0$), and $O$ is the origin $(0, 0)$. Let $s$ represent the distance of any point on $L$ from the origin tracing along the boundary $L$ in the counterclockwise direction. Notice that $L=L_1\uu L_2$,
where
\begin{align*}
L_1&=\set{(x_1, x_2) : x_1=t, \, x_2=0 \te{ for } -1\leq t\leq 1}, \te{ and }\\
L_2&=\set{(x_1, x_2) : x_1=\cos t, \, x_2=\sin t\te{ for } 0\leq t\leq \pi}.
\end{align*}
Let $P$ be the mixed uniform distribution defined on the boundary of the semicircular disc such that $P:=\frac 12 P_1+\frac 12 P_2$, where $P_1$ is the uniform distribution on the base $L_1$ of the semicircular disc, and $P_2$ is the uniform distribution on the semicircular arc $L_2$. Thus, if $f_1$ and $f_2$ are the probability density functions for $P_1$ and $P_2$, we have
\begin{align*}
 f_1(x_1, x_2)=\left\{\begin{array}{cc}
\frac 1 2 & \te{ if } (x_1, x_2) \in L_1,\\
0 & \te{ otherwise,}
\end{array}\right. \te{ and }  f_2(x_1, x_2)=\left\{\begin{array}{cc}
\frac  1 {2 \pi} & \te{ if }  (x_1, x_2) \in L_2,\\
0 & \te{ otherwise.}
\end{array}\right.
\end{align*}
Thus, if $f$ is the probability density function for the mixed distribution $P$, then we have $f=\frac 12 f_1+\frac 12 f_2$, i.e., $f$ is defined by
\begin{align*} \label{eq900}
 f(x_1, x_2)=\left\{\begin{array}{cc}
\frac 1 4 & \te{ if } (x_1, x_2) \in L_1,\\
\frac  1 {2 \pi} & \te{ if }  (x_1, x_2) \in L_2, \\
  0 & \te{ otherwise}.
\end{array}\right.
\end{align*}
On both $L_1$ and $L_2$, we have $ds=\sqrt{(\frac {dx_1}{dt})^2 +(\frac {dx_2}{dt})^2} \, dt=dt$ yielding $dP(s)=P(ds)=f(x_1, x_2)ds=f(x_1, x_2) dt$.

\begin{lemma}
Let $X$ be a continuous random variable with the mixed distribution $P$ taking values on $L$. Then,
\[E(X)=(0, \frac 1 {\pi}) \te{ and } V:=V(X)=\frac{2}{3}-\frac{1}{\pi ^2}.\]
\end{lemma}
\begin{proof} We have,
\begin{align*} \label{eq1}
&E(X)  =\int_L(x_1 i+x_2 j) dP=\frac 1 4\int_{L_1}  (t, 0) \,dt +\frac 1{2 \pi} \int_{L_2}(\cos t, \sin t) \,dt=(0,\frac{1}{\pi }).
\end{align*}  To calculate the variance, we proceed as follows:
\begin{align*} V(X)&=E\|X-E(X)\|^2=\int_L \Big((x_1-0)^2 +(x_2-\frac 1 {\pi})^2\Big)dP\\
&=\frac 1 4 \int_{-1}^1 \Big((t-0)^2 +(0-\frac 1 {\pi})^2\Big)\, dt+\frac 1 {2 \pi} \int_{0}^{\pi}\Big((\cos t-0)^2 +(\sin t-\frac 1 {\pi})^2\Big)dt\\
&=\frac{2}{3}-\frac{1}{\pi ^2}.
\end{align*}
Hence the lemma.
\end{proof}
\begin{remark} \label{remark2} Proceeding similarly as Remark~\ref{remark45}, we see that the optimal set of one-mean is the set $\set{(0, \frac 1 {\pi})}$, and the corresponding quantization error is the variance $V:=V(X)$ of the random variable $X$.
\end{remark}

In the following proposition we give the optimal sets of two-means.
\begin{prop}
Let $P$ be the mixed distribution on the boundary of the semicircle. Then, the set $\set{(-\frac{1}{4}-\frac{1}{\pi },\frac{1}{\pi }), (\frac{1}{4}+\frac{1}{\pi },\frac{1}{\pi })}$ forms the optimal set of two-means, and the quantization error for two-means is given by
\[V_2=2 \left(\frac{96-24 \pi +7 \pi ^2}{192 \pi ^2}+\frac{-96-8 \pi +17 \pi ^2}{64 \pi ^2}\right)=0.242369.\]
\end{prop}
\begin{proof}
 Let the points $P$ and $Q$ form an optimal set of two-means. Let $\ell$ be the boundary of their Voronoi regions. The following two cases can arise:

Case~1.  $\ell$ intersects both $L_1$ and $L_2$.

Let $\ell$ intersect $L_1$ and $L_2$ at the points $D$ and $E$, respectively. Let the points $D$ and $E$ be given by the parameters  $t=\ga$ and $t=\gb$.
Let $P$ and $Q$ be the conditional expectations of the random variable $X$ given that $X$ takes values on the boundaries $DB\uu \arc{BE}$ and $\arc {EA}\uu AD$, respectively. Then, after some calculations, we have
\begin{align*}
\tilde p&=E(X : X\in DB\uu \arc{BE})=\frac{\frac{1}{4} \int_{\alpha }^1 (t, 0) \, dt+\frac 1{2 \pi}\int_0^{\beta } (\cos t, \sin t) \, dt}{\frac{1}{4} \int_{\alpha }^1 1 \, dt+\frac 1{2 \pi}\int_0^{\beta } 1 \, dt}\\
&=\Big(\frac{\frac{1}{4} (\frac{1}{2}-\frac{\alpha ^2}{2} )+\frac{\sin \gb}{2 \pi }}{\frac{1-\alpha }{4}+\frac{\beta }{2 \pi }},\frac{1-\cos \gb}{2 \pi  (\frac{1-\alpha }{4}+\frac{\beta }{2 \pi } )}\Big),
\end{align*}
and similarly \[\tilde q=E(X : X  \in \,\arc{EA}\uu AD)=\Big(\frac{\frac{1}{4}  (\frac{\alpha ^2}{2}-\frac{1}{2} )-\frac{\sin \gb}{2 \pi }}{\frac{\alpha +1}{4}+\frac{\pi -\beta }{2 \pi }},\frac{\cos \gb+1}{2 \pi  (\frac{\alpha +1}{4}+\frac{\pi -\beta }{2 \pi } )}\Big).\]
Since $P$ and $Q$ form an optimal set of two-means, and  $DE$ is the boundary of their corresponding Voronoi regions,  we have the canonical equations as
 $\rho(\tilde d, \tilde p)-\rho(\tilde d, \tilde q)=0$, and $\rho(\tilde e, \tilde p)-\rho(\tilde e, \tilde q)=0$.
Put the values of $\tilde p, \tilde q, \tilde d$ and $\tilde e$, and then solving the two equations in $\ga$ and $\gb$, we have $\ga=0$ and $\gb=\frac \pi 2$ implying
\[\tilde p=(\frac{1}{4}+\frac{1}{\pi },\frac{1}{\pi }), \te{ and } \tilde q=(-\frac{1}{4}-\frac{1}{\pi },\frac{1}{\pi }),\]
and the corresponding distortion $V_2(\te{Case~1})$ error, due to a symmetry, is given by
\begin{align*} V_2(\te{Case~1})&=2 \Big( \frac{1}{4} \int_{\alpha }^1 \rho((t, 0), \tilde p) \, dt+\frac 1{2 \pi}\int_{\beta }^1  \rho((\cos t, \sin t), \tilde p) \, dt\Big)\\
&=2 \Big(\frac{96-24 \pi +7 \pi ^2}{192 \pi ^2}+\frac{-96-8 \pi +17 \pi ^2}{64 \pi ^2}\Big)=0.242369.
\end{align*}

Case~2.  $\ell$ intersects $L_2$ at two points.

Let $\ell$ intersect $L_2$ at the points $D$ and $E$, respectively. As before, there exist parameters $t=\ga$ and $t=\gb$, for which we have
\[\tilde d=(\cos \ga, \sin \ga), \te{ and }\tilde e=(\cos \gb, \sin \gb).\]
Let $P$ and $Q$ be the conditional expectations of the random variable $X$ given that $X$ takes values on the boundary above and below the line $\ell$, respectively. Then,
\[\tilde p=E(X : X \in \arc{DE}), \te{ and } \tilde q=E(X : X\in \arc{EA} \uu AB\uu \arc {BD}).\]
Proceeding in the similar way as Case~1, we calculate $\tilde p$ and $\tilde q$, and obtain the canonical equations
$\rho(\tilde d, \tilde p)-\rho(\tilde d, \tilde q)=0$, and $\rho(\tilde e, \tilde p)-\rho(\tilde e, \tilde q)=0.$
Solving the above two equations in $\ga$ and $\gb$, we have $\ga=0.436587$, and $\gb=\pi-\ga$, i.e., the line $\ell$ is parallel to the base $AOB$ of the semicircle, yielding
\[\tilde p=\Big(0,\frac{2 \cos \ga }{\pi -2 \ga }\Big)=(0, 0.798971), \te{ and } \tilde q=\Big(0,\frac{1-\cos \ga}{\pi   (\frac{\ga }{\pi }+\frac{1}{2})}\Big)=(0, 0.0467274),\]
and the corresponding distortion $V_2(\te{Case~2})$ is given by
\begin{align*} V_2(\te{Case~2})=2\Big(\frac 1 {2\pi} \int_{\gb}^{\frac{\pi }{2}} \rho((\cos t,\sin t), \tilde p)\, dt+\frac{1}{4} \int_0^1 \rho((t,0), \tilde q)dt+\frac 1 {2\pi} \int_{0}^{\gb} \rho((\cos t,\sin t), \tilde q)\, dt\Big)
\end{align*}
implying $V_2(\te{Case~2})=0.434806$.

Since $V_2(\te{Case~2})>V_2(\te{Case~1})$, the points in Case~1 form the optimal set of two-means, and $V_2(\te{Case~1})$ is the quantization error for two-means (see Figure~\ref{Fig3}). Thus, the proof of the proposition is complete.
\end{proof}

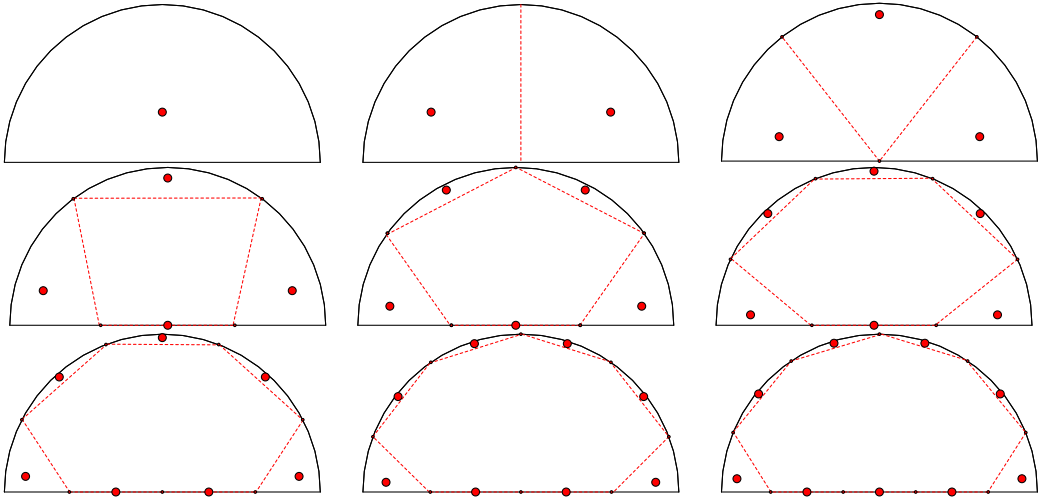
\begin{figure}
\begin{tikzpicture}[line cap=round,line join=round,>=triangle 45,x=0.7cm,y=0.7cm]
\draw [shift={(0.,0.)}] plot[domain=0.:3.141592653589793,variable=\t]({1.*3.*cos(\t r)+0.*3.*sin(\t r)},{0.*3.*cos(\t r)+1.*3.*sin(\t r)});
\draw [shift={(0.,0.)}] plot[domain=0.:3.141592653589793,variable=\t]({1.*3.*cos(\t r)+0.*3.*sin(\t r)},{0.*3.*cos(\t r)+1.*3.*sin(\t r)});
\draw (-3.,0.)-- (3.,0.);
\begin{scriptsize}
\draw [fill=ffqqqq] (0.,0.954929658551372) circle (1.5pt);
\end{scriptsize}
\end{tikzpicture} \quad
\begin{tikzpicture}[line cap=round,line join=round,>=triangle 45,x=0.7cm,y=0.7cm]
\draw [shift={(0.,0.)}] plot[domain=0.:3.141592653589793,variable=\t]({1.*3.*cos(\t r)+0.*3.*sin(\t r)},{0.*3.*cos(\t r)+1.*3.*sin(\t r)});
\draw [shift={(0.,0.)}] plot[domain=0.:3.141592653589793,variable=\t]({1.*3.*cos(\t r)+0.*3.*sin(\t r)},{0.*3.*cos(\t r)+1.*3.*sin(\t r)});
\draw (-3.,0.)-- (3.,0.);
\draw [dash pattern=on 1pt off 1pt,color=ffqqqq] (0.,3.)-- (0.,0.);
\begin{scriptsize}
\draw [fill=ffqqqq] (1.704929658551372,0.954929658551372) circle (1.5pt);
\draw [fill=ffqqqq] (-1.704929658551372,0.954929658551372) circle (1.5pt);
\end{scriptsize}
\end{tikzpicture} \quad
\begin{tikzpicture}[line cap=round,line join=round,>=triangle 45,x=0.7cm,y=0.7cm]
\draw [shift={(0.,0.)}] plot[domain=0.:3.141592653589793,variable=\t]({1.*3.*cos(\t r)+0.*3.*sin(\t r)},{0.*3.*cos(\t r)+1.*3.*sin(\t r)});
\draw [shift={(0.,0.)}] plot[domain=0.:3.141592653589793,variable=\t]({1.*3.*cos(\t r)+0.*3.*sin(\t r)},{0.*3.*cos(\t r)+1.*3.*sin(\t r)});
\draw (-3.,0.)-- (3.,0.);
\draw [dash pattern=on 1pt off 1pt,color=ffqqqq] (1.85038,2.36137)-- (0.,0.);
\draw [dash pattern=on 1pt off 1pt,color=ffqqqq] (0.,0.)-- (-1.85038,2.36137);
\begin{scriptsize}
\draw [fill=ffqqqq] (1.9046,0.46413) circle (1.5pt);
\draw [fill=ffqqqq] (0.,2.78394) circle (1.5pt);
\draw [fill=ffqqqq] (-1.9046,0.46413) circle (1.5pt);
\draw [fill=ffqqqq] (1.85038,2.36137) circle (0.5pt);
\draw [fill=ffqqqq] (0.,0.) circle (0.5pt);
\draw [fill=ffqqqq] (-1.85038,2.36137) circle (0.5pt);
\end{scriptsize}
\end{tikzpicture} \quad
\begin{tikzpicture}[line cap=round,line join=round,>=triangle 45,x=0.7cm,y=0.7cm]
\draw [shift={(0.,0.)}] plot[domain=0.:3.141592653589793,variable=\t]({1.*3.*cos(\t r)+0.*3.*sin(\t r)},{0.*3.*cos(\t r)+1.*3.*sin(\t r)});
\draw [shift={(0.,0.)}] plot[domain=0.:3.141592653589793,variable=\t]({1.*3.*cos(\t r)+0.*3.*sin(\t r)},{0.*3.*cos(\t r)+1.*3.*sin(\t r)});
\draw (-3.,0.)-- (3.,0.);
\draw [line width=0.4pt,dash pattern=on 1pt off 1pt,color=ffqqqq] (-1.7809468457106679,2.41417239085224)-- (1.7740909908825524,2.4192149875671163);
\draw [line width=0.4pt,dash pattern=on 1pt off 1pt,color=ffqqqq] (1.7740909908825524,2.4192149875671163)-- (1.24,0.);
\draw [line width=0.4pt,dash pattern=on 1pt off 1pt,color=ffqqqq] (1.24,0.)-- (-1.3,0.);
\draw [line width=0.4pt,dash pattern=on 1pt off 1pt,color=ffqqqq] (-1.3,0.)-- (-1.7809468457106679,2.41417239085224);
\begin{scriptsize}
\draw [fill=ffqqqq] (0.,0.) circle (1.5pt);
\draw [fill=ffqqqq] (2.3647,0.657512) circle (1.5pt);
\draw [fill=ffqqqq] (-2.3647,0.657512) circle (1.5pt);
\draw [fill=ffqqqq] (0.,2.79861) circle (1.5pt);
\draw [fill=ffqqqq] (1.27376,0.) circle (0.5pt);
\draw [fill=ffqqqq] (-1.27376,0.) circle (0.5pt);
\draw [fill=ffqqqq] (1.79449,2.40412) circle (0.5pt);
\draw [fill=ffqqqq] (-1.79449,2.40412) circle (0.5pt);
\end{scriptsize}
\end{tikzpicture}\quad
\begin{tikzpicture}[line cap=round,line join=round,>=triangle 45,x=0.7cm,y=0.7cm]
\draw [shift={(0.,0.)}] plot[domain=0.:3.141592653589793,variable=\t]({1.*3.*cos(\t r)+0.*3.*sin(\t r)},{0.*3.*cos(\t r)+1.*3.*sin(\t r)});
\draw [shift={(0.,0.)}] plot[domain=0.:3.141592653589793,variable=\t]({1.*3.*cos(\t r)+0.*3.*sin(\t r)},{0.*3.*cos(\t r)+1.*3.*sin(\t r)});
\draw (-3.,0.)-- (3.,0.);
\draw [line width=0.4pt,dash pattern=on 1pt off 1pt,color=ffqqqq] (0.008091676567611426,2.9999890874418735)-- (-2.457975330433301,1.719987579891583);
\draw [line width=0.4pt,dash pattern=on 1pt off 1pt,color=ffqqqq] (-2.457975330433301,1.719987579891583)-- (-1.246,0.);
\draw [line width=0.4pt,dash pattern=on 1pt off 1pt,color=ffqqqq] (-1.246,0.)-- (1.218,0.);
\draw [line width=0.4pt,dash pattern=on 1pt off 1pt,color=ffqqqq] (1.218,0.)-- (2.4413362836008554,1.7435243475139555);
\draw [line width=0.4pt,dash pattern=on 1pt off 1pt,color=ffqqqq] (2.4413362836008554,1.7435243475139555)-- (0.008091676567611426,2.9999890874418735);
\begin{scriptsize}
\draw [fill=ffqqqq] (0.,0.) circle (1.5pt);
\draw [fill=ffqqqq] (2.39157,0.362301) circle (1.5pt);
\draw [fill=ffqqqq] (1.31911,2.57007) circle (1.5pt);
\draw [fill=ffqqqq] (-2.39157,0.362301) circle (1.5pt);
\draw [fill=ffqqqq] (-1.31911,2.57007) circle (1.5pt);
\draw [fill=qqwuqq] (1.22323,0.) circle (0.5pt);
\draw [fill=qqwuqq] (2.43745,1.74895) circle (0.5pt);
\draw [fill=qqwuqq] (-1.22323,0.) circle (0.5pt);
\draw [fill=qqwuqq] (-2.43745,1.74895) circle (0.5pt);
\draw [fill=ffqqqq] (1.22323,0.) circle (0.5pt);
\draw [fill=ffqqqq] (2.43745,1.74895) circle (0.5pt);
\draw [fill=ffqqqq] (0.,3.) circle (0.5pt);
\draw [fill=ffqqqq] (-1.22323,0.) circle (0.5pt);
\draw [fill=ffqqqq] (-2.43745,1.74895) circle (0.5pt);
\end{scriptsize}
\end{tikzpicture} \quad
\begin{tikzpicture}[line cap=round,line join=round,>=triangle 45,x=0.7cm,y=0.7cm]
\draw [shift={(0.,0.)}] plot[domain=0.:3.141592653589793,variable=\t]({1.*3.*cos(\t r)+0.*3.*sin(\t r)},{0.*3.*cos(\t r)+1.*3.*sin(\t r)});
\draw [shift={(0.,0.)}] plot[domain=0.:3.141592653589793,variable=\t]({1.*3.*cos(\t r)+0.*3.*sin(\t r)},{0.*3.*cos(\t r)+1.*3.*sin(\t r)});
\draw (-3.,0.)-- (3.,0.);
\draw [line width=0.4pt,dash pattern=on 1pt off 1pt,color=ffqqqq] (-1.1312602091546702,2.7785338470465555)-- (1.0902078781120599,2.7948965602509155);
\draw [line width=0.4pt,dash pattern=on 1pt off 1pt,color=ffqqqq] (1.0902078781120599,2.7948965602509155)-- (2.7256224066693897,1.2533884059866534);
\draw [line width=0.4pt,dash pattern=on 1pt off 1pt,color=ffqqqq] (2.7256224066693897,1.2533884059866534)-- (1.16,0.);
\draw [line width=0.4pt,dash pattern=on 1pt off 1pt,color=ffqqqq] (1.16,0.)-- (-1.2,0.);
\draw [line width=0.4pt,dash pattern=on 1pt off 1pt,color=ffqqqq] (-1.2,0.)-- (-2.7180423487301457,1.2697424110856157);
\draw [line width=0.4pt,dash pattern=on 1pt off 1pt,color=ffqqqq] (-2.7180423487301457,1.2697424110856157)-- (-1.1312602091546702,2.7785338470465555);
\begin{scriptsize}
\draw [fill=ffqqqq] (0.,0.) circle (1.5pt);
\draw [fill=ffqqqq] (2.34518,0.198347) circle (1.5pt);
\draw [fill=ffqqqq] (-2.34518,0.198347) circle (1.5pt);
\draw [fill=ffqqqq] (2.01705,2.12291) circle (1.5pt);
\draw [fill=ffqqqq] (0.,2.92835) circle (1.5pt);
\draw [fill=ffqqqq] (-2.01705,2.12291) circle (1.5pt);
\draw [fill=ffqqqq] (-1.18098,0.) circle (0.5pt);
\draw [fill=ffqqqq] (1.18098,0.) circle (0.5pt);
\draw [fill=ffqqqq] (2.72559,1.25346) circle (0.5pt);
\draw [fill=ffqqqq] (1.11253,2.78609) circle (0.5pt);
\draw [fill=ffqqqq] (-1.11253,2.78609) circle (0.5pt);
\draw [fill=ffqqqq] (-2.72559,1.25346) circle (0.5pt);
\end{scriptsize}
\end{tikzpicture}
\quad
\begin{tikzpicture}[line cap=round,line join=round,>=triangle 45,x=0.7cm,y=0.7cm]
\draw [shift={(0.,0.)}] plot[domain=0.:3.141592653589793,variable=\t]({1.*3.*cos(\t r)+0.*3.*sin(\t r)},{0.*3.*cos(\t r)+1.*3.*sin(\t r)});
\draw [shift={(0.,0.)}] plot[domain=0.:3.141592653589793,variable=\t]({1.*3.*cos(\t r)+0.*3.*sin(\t r)},{0.*3.*cos(\t r)+1.*3.*sin(\t r)});
\draw (-3.,0.)-- (3.,0.);
\draw [line width=0.4pt,dash pattern=on 1pt off 1pt,color=ffqqqq] (-2.66296049062134,1.3815358936306195)-- (-1.051154708025566,2.809817392606802);
\draw [line width=0.4pt,dash pattern=on 1pt off 1pt,color=ffqqqq] (-1.051154708025566,2.809817392606802)-- (1.0863678797772924,2.796391394241548);
\draw [line width=0.4pt,dash pattern=on 1pt off 1pt,color=ffqqqq] (1.0863678797772924,2.796391394241548)-- (2.6711231603651164,1.3656870293596084);
\draw [line width=0.4pt,dash pattern=on 1pt off 1pt,color=ffqqqq] (2.6711231603651164,1.3656870293596084)-- (1.78,0.);
\draw [line width=0.4pt,dash pattern=on 1pt off 1pt,color=ffqqqq] (-2.66296049062134,1.3815358936306195)-- (-1.76,0.);
\draw [line width=0.4pt,dash pattern=on 1pt off 1pt,color=ffqqqq] (-1.76,0.)-- (-0.02,0.);
\draw [line width=0.4pt,dash pattern=on 1pt off 1pt,color=ffqqqq] (-0.02,0.)-- (1.78,0.);
\begin{scriptsize}
\draw [fill=ffqqqq] (-0.8828,0.) circle (1.5pt);
\draw [fill=ffqqqq] (0.8828,0.) circle (1.5pt);
\draw [fill=ffqqqq] (2.59703,0.296741) circle (1.5pt);
\draw [fill=ffqqqq] (1.95678,2.18591) circle (1.5pt);
\draw [fill=ffqqqq] (0.,2.9338) circle (1.5pt);
\draw [fill=ffqqqq] (-1.95678,2.18591) circle (1.5pt);
\draw [fill=ffqqqq] (-2.59703,0.296741) circle (1.5pt);
\draw [fill=ffqqqq] (-1.7656,0.) circle (0.5pt);
\draw [fill=ffqqqq] (0.,0.) circle (0.5pt);
\draw [fill=ffqqqq] (1.7656,0.) circle (0.5pt);
\draw [fill=ffqqqq] (2.66708,1.37356) circle (0.5pt);
\draw [fill=ffqqqq] (1.07105,2.80229) circle (0.5pt);
\draw [fill=ffqqqq] (-1.07105,2.80229) circle (0.5pt);
\draw [fill=ffqqqq] (-2.66708,1.37356) circle (0.5pt);
\end{scriptsize}
\end{tikzpicture}
\quad
\begin{tikzpicture}[line cap=round,line join=round,>=triangle 45,x=0.7 cm,y=0.7 cm]
\draw [shift={(0.,0.)}] plot[domain=0.:3.141592653589793,variable=\t]({1.*3.*cos(\t r)+0.*3.*sin(\t r)},{0.*3.*cos(\t r)+1.*3.*sin(\t r)});
\draw [shift={(0.,0.)}] plot[domain=0.:3.141592653589793,variable=\t]({1.*3.*cos(\t r)+0.*3.*sin(\t r)},{0.*3.*cos(\t r)+1.*3.*sin(\t r)});
\draw (-3.,0.)-- (3.,0.);
\draw [line width=0.4pt,dash pattern=on 1pt off 1pt,color=ffqqqq] (-1.7284710608681093,2.4520170863477833)-- (0.02013377475005581,2.999932437758276);
\draw [line width=0.4pt,dash pattern=on 1pt off 1pt,color=ffqqqq] (0.02013377475005581,2.999932437758276)-- (1.6918849384725978,2.4774029456205895);
\draw [line width=0.4pt,dash pattern=on 1pt off 1pt,color=ffqqqq] (1.6918849384725978,2.4774029456205895)-- (2.809817392606802,1.051154708025566);
\draw [line width=0.4pt,dash pattern=on 1pt off 1pt,color=ffqqqq] (2.809817392606802,1.051154708025566)-- (1.76,0.);
\draw [line width=0.4pt,dash pattern=on 1pt off 1pt,color=ffqqqq] (1.76,0.)-- (0.04,0.);
\draw [line width=0.4pt,dash pattern=on 1pt off 1pt,color=ffqqqq] (0.04,0.)-- (-1.74,0.);
\draw [line width=0.4pt,dash pattern=on 1pt off 1pt,color=ffqqqq] (-1.74,0.)-- (-2.809817392606802,1.051154708025566);
\draw [line width=0.4pt,dash pattern=on 1pt off 1pt,color=ffqqqq] (-2.809817392606802,1.051154708025566)-- (-1.7284710608681093,2.4520170863477833);
\begin{scriptsize}
\draw [fill=ffqqqq] (-0.860297,0.) circle (1.5pt);
\draw [fill=ffqqqq] (0.860297,0.) circle (1.5pt);
\draw [fill=ffqqqq] (2.56083,0.184716) circle (1.5pt);
\draw [fill=ffqqqq] (2.33216,1.81341) circle (1.5pt);
\draw [fill=ffqqqq] (0.882326,2.81938) circle (1.5pt);
\draw [fill=ffqqqq] (-0.882326,2.81938) circle (1.5pt);
\draw [fill=ffqqqq] (-2.33216,1.81341) circle (1.5pt);
\draw [fill=ffqqqq] (-2.56083,0.184716) circle (1.5pt);
\draw [fill=ffqqqq] (-1.72059,0.) circle (0.5pt);
\draw [fill=ffqqqq] (0.,0.) circle (0.5pt);
\draw [fill=ffqqqq] (1.72059,0.) circle (0.5pt);
\draw [fill=ffqqqq] (2.8102,1.05013) circle (0.5pt);
\draw [fill=ffqqqq] (1.71021,2.46479) circle (0.5pt);
\draw [fill=ffqqqq] (0.,3.) circle (0.5pt);
\draw [fill=ffqqqq] (-1.71021,2.46479) circle (0.5pt);
\draw [fill=ffqqqq] (-2.8102,1.05013) circle (0.5pt);
\end{scriptsize}
\end{tikzpicture}
 \quad
\begin{tikzpicture}[line cap=round,line join=round,>=triangle 45,x=0.7 cm,y=0.7 cm]
\draw [shift={(0.,0.)}] plot[domain=0.:3.141592653589793,variable=\t]({1.*3.*cos(\t r)+0.*3.*sin(\t r)},{0.*3.*cos(\t r)+1.*3.*sin(\t r)});
\draw [shift={(0.,0.)}] plot[domain=0.:3.141592653589793,variable=\t]({1.*3.*cos(\t r)+0.*3.*sin(\t r)},{0.*3.*cos(\t r)+1.*3.*sin(\t r)});
\draw (-3.,0.)-- (3.,0.);
\draw [line width=0.4pt,dash pattern=on 1pt off 1pt,color=ffqqqq] (-1.6732831645161137,2.4900047091013584)-- (0.04053683939107896,2.9997261149398593);
\draw [line width=0.4pt,dash pattern=on 1pt off 1pt,color=ffqqqq] (0.04053683939107896,2.9997261149398593)-- (1.645530965928074,2.508431350500113);
\draw [line width=0.4pt,dash pattern=on 1pt off 1pt,color=ffqqqq] (1.645530965928074,2.508431350500113)-- (2.7798385302537154,1.1280504180739719);
\draw [line width=0.4pt,dash pattern=on 1pt off 1pt,color=ffqqqq] (2.7798385302537154,1.1280504180739719)-- (2.06,0.);
\draw [line width=0.4pt,dash pattern=on 1pt off 1pt,color=ffqqqq] (2.06,0.)-- (0.68,0.);
\draw [line width=0.4pt,dash pattern=on 1pt off 1pt,color=ffqqqq] (0.68,0.)-- (-0.66,0.);
\draw [line width=0.4pt,dash pattern=on 1pt off 1pt,color=ffqqqq] (-0.66,0.)-- (-2.06,0.);
\draw [line width=0.4pt,dash pattern=on 1pt off 1pt,color=ffqqqq] (-2.06,0.)-- (-2.7756856198367217,1.1382307937459935);
\draw [line width=0.4pt,dash pattern=on 1pt off 1pt,color=ffqqqq] (-2.7756856198367217,1.1382307937459935)-- (-1.6732831645161137,2.4900047091013584);
\begin{scriptsize}
\draw [fill=ffqqqq] (-1.37698,0.) circle (1.5pt);
\draw [fill=ffqqqq] (0.,0.) circle (1.5pt);
\draw [fill=ffqqqq] (1.37698,0.) circle (1.5pt);
\draw [fill=ffqqqq] (2.70617,0.252026) circle (1.5pt);
\draw [fill=ffqqqq] (2.29486,1.86371) circle (1.5pt);
\draw [fill=ffqqqq] (0.863014,2.82754) circle (1.5pt);
\draw [fill=ffqqqq] (-0.863014,2.82754) circle (1.5pt);
\draw [fill=ffqqqq] (-2.29486,1.86371) circle (1.5pt);
\draw [fill=ffqqqq] (-2.70617,0.252026) circle (1.5pt);
\draw [fill=ffqqqq] (-2.06547,0.) circle (0.5pt);
\draw [fill=ffqqqq] (-0.688489,0.) circle (0.5pt);
\draw [fill=ffqqqq] (0.688489,0.) circle (0.5pt);
\draw [fill=ffqqqq] (2.06547,0.) circle (0.5pt);
\draw [fill=ffqqqq] (2.77943,1.12905) circle (0.5pt);
\draw [fill=ffqqqq] (1.67524,2.48869) circle (0.5pt);
\draw [fill=ffqqqq] (0.,3.) circle (0.5pt);
\draw [fill=ffqqqq] (-1.67524,2.48869) circle (0.5pt);
\draw [fill=ffqqqq] (-2.77943,1.12905) circle (0.5pt);
\end{scriptsize}
\end{tikzpicture}
\caption{Points in an optimal set of $n$-means for $1\leq n\leq 9$.} \label{Fig3}
\end{figure}

\begin{prop} The set $\set{(-0.634868, 0.15471), (0, 0.92798), (0.634868, 0.15471)}$ forms an optimal set of three-means, and the quantization error for three-means is given by
$V_3=0.147821.$
\end{prop}

\begin{proof} Let the set $\ga:=\set{p, q, r}$ be an optimal set of three-means. Two cases can arise:

Case~1. $\ga$ contains a point from the base $AOB$.

Recall that $P$ is defined by $P=\frac 12 P_1+\frac 12 P_2$, i.e., the probability distribution $P$ is the uniform mixture of two uniform distributions, and the semicircle is symmetric about its vertical axis. Thus, if $\ga$ contains a point from the base, we can assume that $O (0, 0)\in\ga$, and the boundaries of the other two points cut the boundary of the semicircle at the points $F(-a, 0)$, $G(a, 0)$, and $H(0, 1)$, where $0<a<1$. Thus, we can assume that $p=E(X : X\in FG)=(0, 0)$, and $q=E(X : X \in GB\uu\arc {BH})$, and $r=E(X : X \in \arc {HA}\uu AF).$
Solving the canonical equation $\rho(\tl g, p)-\rho(\tl g, q)=0$, we obtain $a=0.462946$ yielding $q=(0.669762, 0.414182)$, and $r=(-0.669762, 0.414182)$. If $V_3(\te{Case~1})$ is the corresponding distortion error, we have
\begin{align*}
V_3(\te{Case~1})=2 \Big(\frac{1}{4} \int_0^{a} \rho((t,0), p) \, dt+\frac{1}{4} \int_{a}^1 \rho((t,0), q) \, dt+\frac{1}{2 \pi } \int_0^{\frac{\pi }{2}} \rho((\cos t, \sin t), q)\, dt\Big)
\end{align*}
yielding $V_3(\te{Case~1})=0.190082$.

Case~2. $\ga$ does not contain any point from the base $AOB$.

Due to symmetry and the uniform mixture of two uniform distributions, in this case we can assume that the points in $\ga$ cut the boundary of the semicircle at the points $O(0, 0)$, $G(\cos b, \sin b)$, and $H(-\cos b, \sin b)$, where $0<b<\frac {\pi} 2$. Thus, we can assume that $p=E(X : X\in OB\uu \arc{BG})$, and $q=E(X : X \in \arc{GH})$, and $r=E(X : X \in \arc {HA}\uu AO).$
Solving the canonical equation $\rho(\tl g, p)-\rho(\tl g, q)=0$, we obtain $b=0.906133$ yielding $p=(0.634868, 0.15471)$, $q=(0, 0.92798)$, and $r=(-0.634868, 0.15471)$. If $V_3(\te{Case~2})$ is the corresponding distortion error, we have
\begin{align*}
V_3(\te{Case~2})=2 \Big(\frac{1}{4} \int_0^{1} \rho((t,0), p) \, dt+\frac{1}{2\pi} \int_{0}^b \rho((\cos t, \sin t), p) \, dt\Big)+\frac{1}{2 \pi } \int_b^{\pi-b} \rho((\cos t, \sin t), q)\, dt
\end{align*}
yielding $V_3(\te{Case~2})=0.147821$.

Since $V_3(\te{Case~1})>V_3(\te{Case~2})$, the points in Case~2 form the optimal set of three-means, and $V_3(\te{Case~2})$ is the quantization error for three-means (see Figure~\ref{Fig3}). Thus, the proof of the proposition is complete.
\end{proof}

Let us now state the following proposition which gives the optimal sets of $n$-means for $4\leq n\leq 9$ for the mixed distribution on the boundary of the semicircle. The proof follows in the similar way as the previous lemma by considering the different cases.

\begin{prop} \label{prop41} Let $P$ be the mixed distribution on the boundary of the semicircle. Then:

$(i)$ the set $\set{(0, 0), (0.788235, 0.219171), (0, 0.932871), (-0.788235, 0.219171)}$ forms an optimal set of four-means with quantization error
$V_4=0.098412$;

$(ii)$ the set $\set{(0, 0), (0.79719, 0.120767), (0.439705, 0.856689), (-0.439705, 0.856689), \\(-0.79719, 0.120767)}$ forms an optimal set of five-means with quantization error $V_5=0.0654358$;

$(iii)$ the set $\set{(0, 0), (0.781728, 0.0661158), (0.672351, 0.707636), (0, 0.976117), \\ (-0.672351, 0.707636), (-0.781728, 0.0661158)}$ forms an optimal set of six-means with quantization error $V_6=0.0499565$;

$(iv)$ the set $\set{(-0.294267, 0), (0.294267, 0), (0.865678, 0.0989137), (0.65226, 0.728637),\\ (0, 0.977935), (-0.65226, 0.728637), (-0.865678, 0.0989137)}$ forms an optimal set of seven-means with quantization error $V_7=0.0366668$;

$(v)$ the set $\set{(-0.286766, 0), (0.286766, 0), (0.853609, 0.0615721), (0.777386,
0.604469), \\ (0.294109, 0.939793), (-0.294109, 0.939793), (-0.777386,
0.604469), \\(-0.853609, 0.0615721)}$ forms an optimal set of eight-means with quantization error $V_8=0.0290573$;

$(vi)$ the set $\set{(-0.458992, 0), (0, 0), (0.458992, 0), (0.902056, 0.0840085),
(0.764954, 0.621235),\\ (0.287671, 0.942514), (-0.287671, 0.942514),
(-0.764954, 0.621235),  (-0.902056, 0.0840085)}$ \\forms an optimal set of nine-means with quantization error $V_9=0.0233983$ (see Figure~\ref{Fig3}).
\end{prop}

\begin{remark} \label{remark11}
Let $\ga_n$ be an optimal set of $n$-means for $P$ for $n\geq 4$. From the above proposition we see that for $4\leq n\leq 9$ there exists a positive integers $k$, and two positive numbers $a$ and $b$ depending on $k$, such that if $m=n-k-2$, then $\ga_n$ contains $k$ quantizers which occur due to the uniform distribution on the closed interval $[-a, a]$, and $m$ quantizers  which occur due to the uniform distribution on the semicircular arc $\set{(\cos \gq, \sin \gq) : b\leq \gq\leq \pi-b}$. It can be proved that this fact is also true for any positive integer $n\geq 10$. In addition, $\ga_n$ contains two quantizers which are in the interior of the angles formed by the base $AOB$ and the semicircular arc $\arc{BA}$, the Voronoi regions of these two points contain elements from both the base and the semicircular arc. Due to too much technicality, we skip the proof of it in the paper. The two real numbers $a$ and $b$ are obtained by solving the two canonical equations
\begin{equation} \label{eq000} \rho(\tl d, \tl p)-\rho(\tl d, \tl q)=0, \te{ and } \rho(\tl e, \tl q)-\rho(\tl e, \tl r)=0,
\end{equation} where
\begin{equation*}
\left\{\begin{array}{ll}
&\tl d=(a, 0), \, e  = (\cos b, \sin b), \, \tl p=(a-\frac{a}{k}, 0), \\
& \tl q=\frac {\frac{1}{4} \int_a^1(t, 0) \, dt+\frac 1{2 \pi} \int_0^{b} (\cos t,\sin t) \, dt}{\frac{1}{4} \int_a^11 \, dt+\frac 1{2 \pi} \int_0^{b} 1 \, dt}=\Big(\frac{-\pi  a^2+4 \sin b+\pi }{8 \pi  \left(\frac{1-a}{4}+\frac{b}{2 \pi }\right)},\frac{\sin ^2\left(\frac{b}{2}\right)}{\pi  \left(\frac{1-a}{4}+\frac{b}{2 \pi }\right)}\Big),  \te{ and }\\
 & \tl r=\frac{\int_b^{\frac{b (m-2)+\pi }{m}} (\cos t,\sin t) \, dt}{\int_b^{\frac{b (m-2)+\pi }{m}} 1 \, dt}=\Big(\frac{m  (\sin \left(\frac{b (m-2)+\pi }{m}\right)-\sin b)}{\pi -2 b},\frac{m(\cos b-\cos (\frac{b (m-2)+\pi }{m}))}{\pi -2 b}\Big).
\end{array}
\right.
\end{equation*}
\end{remark}

Let us now give the following theorem.

\begin{theorem} \label{Th61}
Let $\ga_n$ be an optimal set of $n$-means for $n\geq 4$ such that $\ga_n$ contains $k:=k(n)$ elements from the base of the semicircular disc. Then,
\begin{align*}
\ga_n:&=\set{(-a+\frac{2j-1}{k} a, 0) : 1\leq j\leq k}\\
&\uu \Big\{ \frac {2m}{\pi-2b} \sin(\frac{\pi-2b}{2m}) \Big(\cos (b+(2j-1)\frac{\pi-2b}{2m}), \   \sin (b+(2j-1)\frac{\pi-2b}{2m})\Big) : 1\leq j\leq m\Big\}\\
&\uu \set{(r, s), (-r, s)},
\end{align*}
where $m=n-k-2$,  $r=\frac{-\pi  a^2+4 \sin b+\pi }{8 \pi  (\frac{1-a}{4}+\frac{b}{2 \pi })}$, $s=\frac{\sin ^2 (\frac{b}{2})}{\pi  (\frac{1-a}{4}+\frac{b}{2 \pi })}$, and the two positive real numbers $a$ and $b$ are determined by the equations in \eqref{eq000},
and the quantization error for $n$-means is given by
\begin{align*}\label{eq34}
 V_n&= \frac{1}{24}\Big(\frac{4 a^3}{k^2}+\frac{12  (2 m^2 \cos (\frac{\pi -2 b}{m})+(\pi -2 b)^2-2 m^2)}{\pi  (\pi -2 b)^2} \notag \\
&-\frac 1 {\pi  (\pi  (a-1)-2 b)}\Big(\pi ^2 a^4-8 \pi  a^3 b-4 \pi ^2 a^3+24 \pi  (a^2-1) \sin b+6 \pi ^2 a^2-24 \pi  a b \\
&-4 \pi ^2 a+48 b^2+32 \pi  b+96 \cos b+\pi ^2-96\Big)\Big). \notag
\end{align*}
\end{theorem}

\begin{proof}
Let $\gg_n$ be the set of $k:=k(n)$ quantizers that $\ga_n$ contains from the closed interval $[-a, a]$. Then, by \cite{RR2},
\[\gg_n=\set{(-a+\frac{2i-1}{k} a, 0) : 1\leq i\leq k},\]
and the distortion error due to the set $\gg_n$ is given by $\frac{a^3}{6 k^2}$. Let $\gd_n$ be the set of $m$ quantizers  which occur due to the uniform distribution on the circular arc $\set{(\cos \gq, \sin \gq) : b\leq \gq\leq \pi-b}$. Then, by Proposition~\ref{prop32}, we have
\[\gd_n:=\Big\{ \frac {2m}{\pi-2b} \sin(\frac{\pi-2b}{2m}) \Big(\cos (b+(2j-1)\frac{\pi-2b}{2m}), \   \sin (b+(2j-1)\frac{\pi-2b}{2m})\Big) : 1\leq j\leq m\Big\},\]
and the corresponding distortion error is given by
\begin{align*} &\frac m {2 \pi } \int_b^{b+\frac{\pi -2 b}{m}} \rho\Big((\cos t, \sin t), \frac {2m}{\pi-2b} \sin(\frac{\pi-2b}{2m})  \Big(\cos (b+ \frac{\pi-2b}{2m}), \   \sin (b+ \frac{\pi-2b}{2m})\Big)\Big) \, dt\\
&=\frac{2 m^2 \cos \left(\frac{\pi -2 b}{m}\right)+(\pi -2 b)^2-2 m^2}{2 \pi  (\pi -2 b)^2}.\end{align*}
As mentioned in Remark~\ref{remark11}, let $(r, s)$ be the point in $\ga_n$ which lies in the interior of the right hand angle formed by the base and the semicircular arc. Then,
\begin{align*} (r, s)&=\frac {\frac{1}{4} \int_a^1(t, 0) \, dt+\frac 1{2 \pi} \int_0^{b} (\cos t,\sin t) \, dt}{\frac{1}{4} \int_a^11 \, dt+\frac 1{2 \pi} \int_0^{b} 1 \, dt}=\Big(\frac{-\pi  a^2+4 \sin b+\pi }{8 \pi  (\frac{1-a}{4}+\frac{b}{2 \pi })},\frac{\sin ^2(\frac{b}{2})}{\pi  (\frac{1-a}{4}+\frac{b}{2 \pi })}\Big).
\end{align*}
Due to symmetry, the point in $\ga_n$ which lies in the interior of the left hand angle formed by the base and the semicircular arc is given by $(-r, s)$, where
$r=\frac{-\pi  a^2+4 \sin b+\pi }{8 \pi  (\frac{1-a}{4}+\frac{b}{2 \pi })}$, and $s=\frac{\sin ^2(\frac{b}{2})}{\pi  (\frac{1-a}{4}+\frac{b}{2 \pi })}$.
The distortion error due to these two corner points is given by
\begin{align*} &2  \Big(\frac{1}{4} \int_a^1 \rho((t,0),(r, s)) \, dt+\frac 1{2\pi} \int_0^b \rho((\cos t,\sin t), (r, s)) \, dt\Big)\\
&=-\frac 1 {24\pi  (\pi  (a-1)-2 b)}(\pi ^2 a^4-8 \pi  a^3 b-4 \pi ^2 a^3+24 \pi  (a^2-1) \sin b+6 \pi ^2 a^2-24 \pi  a b\\
&-4 \pi ^2 a+48 b^2+32 \pi  b+96 \cos b+\pi ^2-96)
\end{align*}
Taking the union of $\gg_n$, $\gd_n$, and the set $\set{(r, s), (-r, s)}$, we obtain $\ga_n$, and summing up the corresponding distortion errors, we obtain the quantization error $V_n$. Thus, the proof of the theorem is complete.
\end{proof}

\begin{remark}
Let $n\in \D N$ be such that $n\geq 4$. Then, the positive integer $k$, given in Theorem~\ref{Th61}, depends on $n$. If $k$ is known, using Theorem~\ref{Th61}, one can easily determine the optimal set $\ga_n$ and the corresponding quantization error.
\end{remark}
Let us now give the following definition.

\begin{defi} \label{difi21}
Define the sequence $\set{a(n)}$ such that $a(n)=\lfloor n(\sqrt 2-1)\rfloor$ for $n\geq 1$, i.e.,
\begin{align*}
 \set{a(n)}_{n=1}^\infty=&\set{0, 0, 1, 1,2,2,2,3,3,4,4,4,5,5,6,6,7,7,7,8,8,9,9,9,10,10,11,11,12,12, \cdots},
\end{align*}
where $\lfloor x\rfloor$ represents the greatest integer not exceeding $x$.
\end{defi}

The following algorithm helps us to determine the exact value of $k$ mentioned in Theorem~\ref{Th61}.
\subsection{Algorithm}  Let $n\geq 4$, and let $V(n, k):=V_n$, as given by Theorem~\ref{Th61}, denote the distortion error if $\ga_n$ contains $k$ elements from the base of the semicircular disc. Let $\set{a(n)}$ be the sequence  defined by Definition~\ref{difi21}. Then, the algorithm runs as follows:

$(i)$ Write $k:=a(n)$.

$(ii)$ If $k=1$ go to step $(v)$, else step $(iii)$.

$(iii)$ If $V(n, k-1)<V(n, k)$ replace $k$ by $k-1$ and go to step $(ii)$, else step $(iv)$.

$(iv)$ If $V(n, k+1)<V(n, k)$ replace $k$ by $k+1$ and return, else step $(v)$.

$(v)$ End.

When the algorithm ends, then the value of $k$, obtained, is the exact value of $k$ that an optimal set $\ga_n$ contains from the base of the semicircular disc.
\begin{remark}
If $n=40$, then $a(n)=16$, and by the algorithm we also obtain $k=16$; if $n=51$, then $a(n)=21$, and by the algorithm we also obtain $k=21$. If $n=1000$, then  $a(n)=414$, and by the algorithm, we obtain $k=424$; if $n=2500$, then $a(n)=1035$, and by the algorithm, we obtain $k=1042$; and if $n=5000$, then $a(n)=2071$, and by the algorithm, we obtain $k=2083$. Thus, we see that with the help of the sequence and the algorithm we can easily determine the exact value of $k$ for any positive integer $n\geq 4$.
\end{remark}

The following questions still remain open.
\subsection{Open} Can one estimate the positive integer $k$ in terms of $n$ for all large $n$? Further, from such estimate can one obtain the asymptotics of the quantization error, and even obtain the existence of quantization dimension and quantization coefficient?

\section{Quantization for a uniform distribution on an elliptical curve} \label{sec3}

The basic definitions and notations that were defined in the first paragraph of Section~\ref{sec1} are also used in this section. As a prototype, we take the equation of the ellipse as $x_1^2+4x_2^2=4$, whose center is $O(0, 0)$, and the lengths of the major and the minor axes are, respectively, $2$ and $1$. By the elliptical curve, denoted by $L$, we mean the boundary of the ellipse $x_1^2+4x_2^2=4$. Let $L$ intersect the positive and negative directions of the $x_1$-axis at the points $A_1$ and $A_3$, and the positive and negative directions of the $x_2$-axis at the points $A_2$ and $A_4$, respectively. Let $P$ be the uniform distribution defined on $L$. Notice that the parametric equations of $L$ are given by $x_1=2 \cos \gq$, and $x_2=\sin \gq$ for $0\leq \gq\leq 2\pi$. Let $s$ represent the distance of any point on $L$ from the point $A_1$ tracing along the boundary $L$ in the counterclockwise direction. Then, $ds=\sqrt{dx_1^2+dx_2^2}=\sqrt{4 \sin ^2\theta +\cos ^2\theta}\,d\gq.$ Thus, the length of $L$ is given by
\[\int_L ds=\int_0^{2 \pi } \sqrt{4 \sin ^2\theta +\cos ^2\theta } \, d\theta=9.6884482205 \te{ (up to ten decimal places)}.\]
In the sequel, write $A:=9.6884482205$.
Hence, the probability density function (pdf) $f$ of the uniform distribution $P$ is given by $f(s):=f(x_1, x_2)=\frac 1{A}$ for all $(x_1, x_2)\in L$, and zero, otherwise.
Again, $dP(s)=P(ds)=f(x_1, x_2) ds=\frac 1{A} \sqrt{4 \sin ^2\theta +\cos ^2\theta } \, d\theta$.

Let us now prove the following lemma.
\begin{lemma} \label{lemma00}
Let $X$ be a continuous random variable with uniform distribution on $L$. Then,
$E(X)=(0, 0) \te{ and } V:=V(X)=2.260230080.$
\end{lemma}
\begin{proof} We have,
\begin{align*}
E(X) &=\int_L(x_1 \, i+x_2 \, j) dP =\frac{1}{A}\int_0^{2 \pi } \sqrt{4 \sin ^2\theta +\cos ^2\theta } (2\cos \gq,\sin \gq)\, d\theta=(0, 0),
\end{align*}
and
\begin{align*} V(X)&=E\|X-E(X)\|^2= \int_L\rho((x_1, x_2), E(X))dP= \frac 1 A \int_{L}\rho((2\cos \gq, \sin \gq), (0, 0))ds\\
&=\frac 1 A \int_0^{2 \pi } \sqrt{4 \sin ^2\gq+\cos ^2\gq} \,\rho((2 \cos \gq,\sin \gq), (0,0))\, d\theta=2.260230080.
\end{align*}
Hence, the proof of the lemma is complete.
\end{proof}

\begin{remark} \label{remark3} Proceeding similarly as Remark~\ref{remark45}, we see that the optimal set of one-mean is the set $\set{(0, 0)}$, and the corresponding quantization error is the variance $V:=V(X)$ of the random variable $X$.
\end{remark}

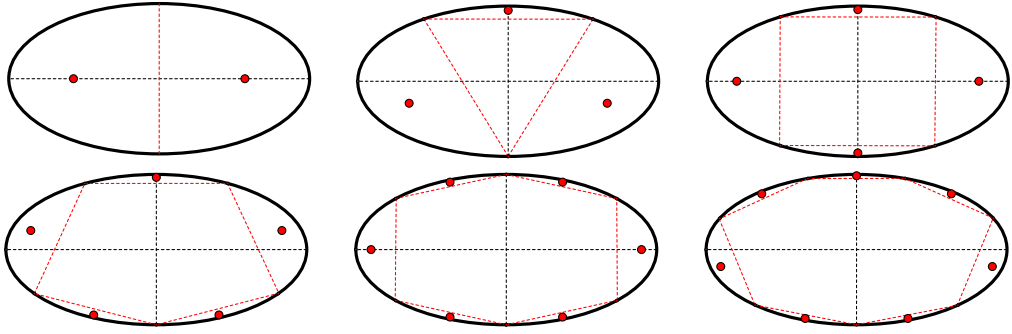
\begin{figure}
\begin{tikzpicture}[line cap=round,line join=round,>=triangle 45,x=1.0cm,y=1.0cm]
\clip(-2.164474304911254,-1.1642922922893124) rectangle (2.1269429504593886,1.1870099205617441);
\draw [rotate around={0.:(0.,0.)},line width=1.2pt] (0.,0.) ellipse (2.cm and 1.cm);
\draw [dash pattern=on 1pt off 1pt] (2.,0.)-- (-2.,0.);
\draw [dash pattern=on 1pt off 1pt,color=ffqqqq] (0.,-1.)-- (0.,1.);
\begin{scriptsize}
\draw [fill=ffqqqq] (-1.13964,0.) circle (1.5pt);
\draw [fill=ffqqqq] (1.13964,0.) circle (1.5pt);
\end{scriptsize}
\end{tikzpicture}\quad
\begin{tikzpicture}[line cap=round,line join=round,>=triangle 45,x=1.0cm,y=1.0cm]
\clip(-2.1008189434403577,-1.1321475309991456) rectangle (2.1397400374673494,1.0742849031370176);
\draw [rotate around={0.:(0.,0.)},line width=1.2pt] (0.,0.) ellipse (2.cm and 1.cm);
\draw [dash pattern=on 1pt off 1pt] (2.,0.)-- (-2.,0.);
\draw [dash pattern=on 1pt off 1pt] (0.,-1.)-- (0.,1.);
\draw [dash pattern=on 1pt off 1pt,color=ffqqqq] (-1.1206946198484502,0.8282577450653177)-- (1.12792,0.825804);
\draw [dash pattern=on 1pt off 1pt,color=ffqqqq] (1.12792,0.825804)-- (0.,-1.);
\draw [dash pattern=on 1pt off 1pt,color=ffqqqq] (0.,-1.)-- (-1.1206946198484502,0.8282577450653177);
\begin{scriptsize}
\draw [fill=ffqqqq] (-2.,0.) circle (0.5pt);
\draw [fill=ffqqqq] (0.,-1.) circle (0.5pt);
\draw [fill=ffqqqq] (0.,0.943319) circle (1.5pt);
\draw [fill=ffqqqq] (-1.3166,-0.29241) circle (1.5pt);
\draw [fill=ffqqqq] (1.3166,-0.29241) circle (1.5pt);
\draw [fill=ffqqqq] (-1.12792,0.825804) circle (0.5pt);
\draw [fill=ffqqqq] (1.12792,0.825804) circle (0.5pt);
\draw [fill=ffqqqq] (0.,-1.) circle (0.5pt);
\draw [fill=ffqqqq] (-1.1206946198484502,0.8282577450653177) circle (0.5pt);
\draw [fill=ffqqqq] (1.12792,0.825804) circle (0.5pt);
\end{scriptsize}
\end{tikzpicture}\quad
\begin{tikzpicture}[line cap=round,line join=round,>=triangle 45,x=1.0cm,y=1.0cm]
\clip(-2.093843171236543,-1.1321475309991456) rectangle (2.146715809671164,1.0742849031370176);
\draw [rotate around={0.:(0.,0.)},line width=1.2pt] (0.,0.) ellipse (2.cm and 1.cm);
\draw [dash pattern=on 1pt off 1pt] (2.,0.)-- (-2.,0.);
\draw [dash pattern=on 1pt off 1pt] (0.,-1.)-- (0.,1.);
\draw [dash pattern=on 1pt off 1pt,color=ffqqqq] (-1.0320353998483365,0.8565779201946379)-- (1.0402355022128402,0.8540945644271484);
\draw [dash pattern=on 1pt off 1pt,color=ffqqqq] (1.0402355022128402,0.8540945644271484)-- (1.0229785844486008,-0.859289650781321);
\draw [dash pattern=on 1pt off 1pt,color=ffqqqq] (-1.0398336759895928,-0.8542168820443043)-- (-1.0320353998483365,0.8565779201946379);
\draw [dash pattern=on 1pt off 1pt,color=ffqqqq] (-1.0398336759895928,-0.8542168820443043)-- (1.0229785844486008,-0.859289650781321);
\begin{scriptsize}
\draw [fill=ffqqqq] (-2.,0.) circle (0.5pt);
\draw [fill=ffqqqq] (0.,-1.) circle (0.5pt);
\draw [fill=ffqqqq] (1.60986,0.) circle (1.5pt);
\draw [fill=ffqqqq] (0.,0.953301) circle (1.5pt);
\draw [fill=ffqqqq] (-1.60986,0.) circle (1.5pt);
\draw [fill=ffqqqq] (0.,-0.953301) circle (1.5pt);
\draw [fill=ffqqqq] (1.03023,0.857121) circle (0.5pt);
\draw [fill=ffqqqq] (1.03023,-0.857121) circle (0.5pt);
\draw [fill=ffqqqq] (-1.03023,0.857121) circle (0.5pt);
\draw [fill=ffqqqq] (-1.03023,-0.857121) circle (0.5pt);
\end{scriptsize}
\end{tikzpicture}\quad
\begin{tikzpicture}[line cap=round,line join=round,>=triangle 45,x=1.0cm,y=1.0cm]
\clip(-2.1008189434403577,-1.1321475309991456) rectangle (2.1397400374673494,1.0742849031370176);
\draw [rotate around={0.:(0.,0.)},line width=1.2pt] (0.,0.) ellipse (2.cm and 1.cm);
\draw [dash pattern=on 1pt off 1pt] (2.,0.)-- (-2.,0.);
\draw [dash pattern=on 1pt off 1pt] (0.,-1.)-- (0.,1.);
\draw [dash pattern=on 1pt off 1pt,color=ffqqqq] (-0.9505294036117456,0.8798428627984517)-- (0.9450959813396136,0.8813049395719558);
\draw [dash pattern=on 1pt off 1pt,color=ffqqqq] (0.9450959813396136,0.8813049395719558)-- (1.6286696475978306,-0.580395377952297);
\draw [dash pattern=on 1pt off 1pt,color=ffqqqq] (1.6286696475978306,-0.580395377952297)-- (0.,-1.);
\draw [dash pattern=on 1pt off 1pt,color=ffqqqq] (0.,-1.)-- (-1.6212490047154264,-0.585566320904225);
\draw [dash pattern=on 1pt off 1pt,color=ffqqqq] (-1.6212490047154264,-0.585566320904225)-- (-0.9505294036117456,0.8798428627984517);
\begin{scriptsize}
\draw [fill=ffqqqq] (-2.,0.) circle (0.5pt);
\draw [fill=ffqqqq] (0.,-1.) circle (0.5pt);
\draw [fill=ffqqqq] (0.,0.960639) circle (1.5pt);
\draw [fill=ffqqqq] (-1.669,0.253562) circle (1.5pt);
\draw [fill=ffqqqq] (-0.832677,-0.869617) circle (1.5pt);
\draw [fill=ffqqqq] (0.832677,-0.869617) circle (1.5pt);
\draw [fill=ffqqqq] (1.669,0.253562) circle (1.5pt);
\draw [fill=ffqqqq] (-0.950097,0.87996) circle (0.5pt);
\draw [fill=ffqqqq] (-1.62245,-0.584733) circle (0.5pt);
\draw [fill=ffqqqq] (0.950097,0.87996) circle (0.5pt);
\draw [fill=ffqqqq] (1.62245,-0.584733) circle (0.5pt);
\end{scriptsize}
\end{tikzpicture}\quad
\begin{tikzpicture}[line cap=round,line join=round,>=triangle 45,x=1.0cm,y=1.0cm]
\clip(-2.1008189434403577,-1.1321475309991456) rectangle (2.1397400374673494,1.0742849031370176);
\draw [rotate around={0.:(0.,0.)},line width=1.2pt] (0.,0.) ellipse (2.cm and 1.cm);
\draw [dash pattern=on 1pt off 1pt] (2.,0.)-- (-2.,0.);
\draw [dash pattern=on 1pt off 1pt] (0.,-1.)-- (0.,1.);
\draw [dash pattern=on 1pt off 1pt,color=ffqqqq] (0.,1.)-- (-1.4633723874508653,0.6816416315862674);
\draw [dash pattern=on 1pt off 1pt,color=ffqqqq] (-1.4633723874508653,0.6816416315862674)-- (-1.4763593261235566,-0.6746041691555127);
\draw [dash pattern=on 1pt off 1pt,color=ffqqqq] (-1.4763593261235566,-0.6746041691555127)-- (0.,-1.);
\draw [dash pattern=on 1pt off 1pt,color=ffqqqq] (0.,-1.)-- (1.4762223220725172,-0.6746791192512974);
\draw [dash pattern=on 1pt off 1pt,color=ffqqqq] (1.4762223220725172,-0.6746791192512974)-- (1.4682492885872152,0.6790147322707986);
\draw [dash pattern=on 1pt off 1pt,color=ffqqqq] (1.4682492885872152,0.6790147322707986)-- (0.,1.);
\begin{scriptsize}
\draw [fill=ffqqqq] (-2.,0.) circle (0.5pt);
\draw [fill=ffqqqq] (0.,-1.) circle (0.5pt);
\draw [fill=ffqqqq] (1.7975,0.) circle (1.5pt);
\draw [fill=ffqqqq] (0.748205,0.897609) circle (1.5pt);
\draw [fill=ffqqqq] (-0.748205,0.897609) circle (1.5pt);
\draw [fill=ffqqqq] (-1.7975,0.) circle (1.5pt);
\draw [fill=ffqqqq] (-0.748205,-0.897609) circle (1.5pt);
\draw [fill=ffqqqq] (0.748205,-0.897609) circle (1.5pt);
\draw [fill=ffqqqq] (1.4693,0.678447) circle (0.5pt);
\draw [fill=ffqqqq] (0.,1.) circle (0.5pt);
\draw [fill=ffqqqq] (1.4693,-0.678447) circle (0.5pt);
\draw [fill=ffqqqq] (0.,-1.) circle (0.5pt);
\draw [fill=ffqqqq] (-1.4693,0.678447) circle (0.5pt);
\draw [fill=ffqqqq] (-1.4693,-0.678447) circle (0.5pt);
\end{scriptsize}
\end{tikzpicture}\quad
\begin{tikzpicture}[line cap=round,line join=round,>=triangle 45,x=1.0cm,y=1.0cm]
\clip(-2.1008189434403577,-1.1321475309991456) rectangle (2.1397400374673494,1.0742849031370176);
\draw [rotate around={0.:(0.,0.)},line width=1.2pt] (0.,0.) ellipse (2.cm and 1.cm);
\draw [dash pattern=on 1pt off 1pt] (2.,0.)-- (-2.,0.);
\draw [dash pattern=on 1pt off 1pt] (0.,-1.)-- (0.,1.);
\draw [dash pattern=on 1pt off 1pt,color=ffqqqq] (-0.6525079538341584,0.9452821497022941)-- (-1.8260324049603411,0.40792329427073454);
\draw [dash pattern=on 1pt off 1pt,color=ffqqqq] (-1.8260324049603411,0.40792329427073454)-- (-1.3502452440318975,-0.7377055274581522);
\draw [dash pattern=on 1pt off 1pt,color=ffqqqq] (-1.3502452440318975,-0.7377055274581522)-- (0.,-1.);
\draw [dash pattern=on 1pt off 1pt,color=ffqqqq] (0.,-1.)-- (1.3385108396349272,-0.74303242395265);
\draw [dash pattern=on 1pt off 1pt,color=ffqqqq] (1.3385108396349272,-0.74303242395265)-- (1.8162943212343567,0.41865108940853957);
\draw [dash pattern=on 1pt off 1pt,color=ffqqqq] (1.8162943212343567,0.41865108940853957)-- (0.6399168715870068,0.9474315803078209);
\draw [dash pattern=on 1pt off 1pt,color=ffqqqq] (0.6399168715870068,0.9474315803078209)-- (-0.6525079538341584,0.9452821497022941);
\begin{scriptsize}
\draw [fill=ffqqqq] (-2.,0.) circle (0.5pt);
\draw [fill=ffqqqq] (0.,0.98224) circle (1.5pt);
\draw [fill=ffqqqq] (-1.26066,0.741926) circle (1.5pt);
\draw [fill=ffqqqq] (-1.80513,-0.225347) circle (1.5pt);
\draw [fill=ffqqqq] (-0.68335,-0.916091) circle (1.5pt);
\draw [fill=ffqqqq] (0.68335,-0.916091) circle (1.5pt);
\draw [fill=ffqqqq] (1.80513,-0.225347) circle (1.5pt);
\draw [fill=ffqqqq] (1.26066,0.741926) circle (1.5pt);
\draw [fill=ffqqqq] (-0.646392,0.946332) circle (0.5pt);
\draw [fill=ffqqqq] (-1.8168,0.418098) circle (0.5pt);
\draw [fill=ffqqqq] (-1.34776,-0.73884) circle (0.5pt);
\draw [fill=ffqqqq] (0.646392,0.946332) circle (0.5pt);
\draw [fill=ffqqqq] (1.8168,0.418098) circle (0.5pt);
\draw [fill=ffqqqq] (1.34776,-0.73884) circle (0.5pt);
\draw [fill=ffqqqq] (0.,-1.) circle (0.5pt);
\end{scriptsize}
\end{tikzpicture}
\caption{Points in an optimal set of $n$-means for $2\leq n\leq 7$.} \label{Fig4}
\end{figure}

\begin{remark} \label{remark90}
The ellipse has two lines of symmetry: the major axis and the minor axis, and the probability distribution is uniform. To calculate the optimal sets of $n$-means for any positive integer $n\geq 2$, we will use this information. This will help us to avoid too much technicality in the proof of the following propositions.
\end{remark}
\begin{prop} \label{prop70}
The optimal set of two-means is $\set{(1.13964, 0), (-1.13964, 0)}$ with quantization error $V_2=0.961441$.
\end{prop}

\begin{proof}
Let $\ga:=\set{\tilde p, \tilde q}$ be an optimal set of two-means. Due to Remark~\ref{remark90}, we can assume that the boundary of their Voronoi regions passes through the center of the ellipse, in other words, we can assume that the boundary of the Voronoi regions cut the ellipse at the two points $D$ and $E$ given by the parameters $\gq=b$ and $\gq=\pi+b$, respectively, where $0\leq b\leq \pi$. Then, we have $\tl d=(2\cos b, \sin b)$, and
\[\tl p=\frac{\int_b^{\pi+b} \sqrt{4 \sin ^2\gq+\cos ^2\gq} (2 \cos \gq,\sin \gq)\, d\theta }{\int_b^{\pi+b} \sqrt{4 \sin ^2\gq+\cos ^2\gq} \, d\theta }, \te{ and } \tl q=\frac{\int_{\pi+b}^{2 \pi+b} \sqrt{4 \sin ^2\gq+\cos ^2\gq} (2 \cos \gq,\sin \gq) \, d\theta }{\int_{\pi+b }^{2 \pi+b } \sqrt{4 \sin ^2\gq+\cos ^2\gq} \, d\theta }.\]
Solving the canonical equation $\rho(\tl d, \tl p)-\rho(\tl d, \tl q)=0$, we have $b=0, \frac \pi 2, \pi$. Notice that $b=0$ and $b=\pi$ are reflections of each other about the origin. Thus, the following two cases can arise:

Case~1. $b=0$.

In this case, we have  $\tl p=(0, 0.705665)$, and $\tl q=(0, -0.705665)$ with the distortion error
\[\int\min_{a\in \ga} \|x-a\|^2 dP=\frac 2 {A}\int_{0}^\pi\rho((2\cos \gq, \sin \gq), (0, 0.705665)) ds=1.76227.\]

Case~2. $b=\frac \pi 2$.

In this case, we have $\tl p=(-1.13964, 0)$, and $\tl q=(1.13964, 0)$ with the distortion error
\[\int\min_{a\in \ga} \|x-a\|^2 dP=\frac 2 {A}\int_{ \frac \pi 2}^{\pi +\frac \pi 2}\rho((2\cos\gq, \sin \gq), (-1.13964, 0)) ds=0.961441.\]

Comparing the distortion errors, we see that the set $\set{(1.13964, 0), (-1.13964, 0)}$ forms the optimal set of two-means  (see Figure~\ref{Fig4}) with quantization error $V_2=0.961441$, which yields the proposition.
\end{proof}

In the following two propositions we state and prove the optimal sets of six- and seven-means.
\begin{prop} \label{prop71}
The optimal set of six-means is
\begin{align*}
\set{(1.7975,0), &(0.748205, 0.897609), (-0.748205, 0.897609), \\
&(-1.7975,0),  (-0.748205, -0.897609), (0.748205, -0.897609)}
\end{align*}
 with quantization error $V_6=0.198794$.
\end{prop}
\begin{proof}
Let $\ga$ be an optimal set of six-means. Due to Remark~\ref{remark90}, the following two cases can arise:

Case~1. Two points of $\ga$ are on the major axis, two are above and two are below the major axis.

Again, due to Remark~\ref{remark90}, the two points which are on the major axis are reflections of each other with respect to the origin, and the set of points below are the reflections of the set of points above with respect to the major axis. Let the point which lie on the positive direction of the major axis be $\tl p$, and the two points above the major axis be given by $\tl q$, and $\tl r$. Let the boundary of the Voronoi regions of $\tl p$ and $\tl q$ be given by $\tl d$ with parameter $\gq=b$, and the boundary of the Voronoi regions of $\tl q$ and $\tl r$ be given by $\tl e$ with parameter $\gq=c$. Then, we have
\begin{align*}
\tl d & =(2 \cos b, \sin b), \ \tl e=(2\cos c, \sin c),  \ \tl p=\frac{\int_{2\pi-b}^{2\pi+b} \sqrt{4 \sin ^2\gq+\cos ^2\gq} (2 \cos \gq,\sin \gq)\, d\theta }{\int_{2\pi-b}^{2\pi+b} \sqrt{4 \sin ^2\gq+\cos ^2\gq} \, d\theta }, \\
\tl q& =\frac{\int_{b}^{c} \sqrt{4 \sin ^2\gq+\cos ^2\gq} (2 \cos \gq,\sin \gq) \, d\theta }{\int_{b}^{c} \sqrt{4 \sin ^2\gq+\cos ^2\gq} \, d\theta }, \te{ and } \tl r =\frac{\int_{c}^{\pi-b} \sqrt{4 \sin ^2\gq+\cos ^2\gq} (2 \cos \gq,\sin \gq) \, d\theta }{\int_{c}^{\pi-b} \sqrt{4 \sin ^2\gq+\cos ^2\gq} \, d\theta }.
\end{align*}
Solving the canonical equations $\rho(\tl d,\tl p)-\rho(\tl d, \tl q)=0$ and $\rho(\tl e, \tl q)-\rho(\tl e, \tl r)=0$, we obtain $b=0.745647$, and $c =1.5708$, yielding \[\set{\tl p, \tl q, \tl r}=\set{(1.7975, 0), (0.748205, 0.897609), (-0.748205, 0.897609)}.\] Thus, due to reflection, we can obtain all the elements of $\ga$, and the corresponding distortion error is given by
\[\int\min_{a\in \ga} \|x-a\|^2 dP=2 (\te{distortion error due to }\tl p, \tl q, \tl r)=0.198794.\]

Case~2. Two points of $\ga$ are on the minor axis, two are to the right and two are to the left of the minor axis.

Due to Remark~\ref{remark90}, the two points which are on the minor axis are reflections of each other with respect to the origin, and the set of points to the right are the reflections of the set of points to the left with respect to the minor axis. Let the point which lie on the positive direction of the minor axis be $\tl p$, and the two points to the left of the minor axis be given by $\tl q$, and $\tl r$. Let the boundary of the Voronoi regions of $\tl p$ and $\tl q$ be given by $\tl d$ with parameter $\gq=b$, and the boundary of the Voronoi regions of $\tl q$ and $\tl r$ be given by $\tl e$ with parameter $\gq=c$. Then, we have
\begin{align*}
\tl d & =(2 \cos b, \sin b), \ \tl e=(2\cos c, \sin c),  \ \tl p=\frac{\int_{\pi-b}^{b} \sqrt{4 \sin ^2\gq+\cos ^2\gq} (2 \cos \gq,\sin \gq)\, d\theta }{\int_{\pi-b}^{b} \sqrt{4 \sin ^2\gq+\cos ^2\gq} \, d\theta }, \\
\tl q& =\frac{\int_{c}^{b} \sqrt{4 \sin ^2\gq+\cos ^2\gq} (2 \cos \gq,\sin \gq) \, d\theta }{\int_{c}^{b} \sqrt{4 \sin ^2\gq+\cos ^2\gq} \, d\theta }, \te{ and } \tl r =\frac{\int_{c}^{2\pi-b} \sqrt{4 \sin ^2\gq+\cos ^2\gq} (2 \cos \gq,\sin \gq) \, d\theta }{\int_{c}^{2\pi-b} \sqrt{4 \sin ^2\gq+\cos ^2\gq} \, d\theta }.
\end{align*}
Solving the canonical equations $\rho(\tl d,\tl p)-\rho(\tl d, \tl q)=0$ and $\rho(\tl e, \tl q)-\rho(\tl e, \tl r)=0$, we obtain $b=1.97683$, and $c =3.14159$, yielding \[\set{\tl p, \tl q, \tl r}=\set{(0, 0.973196), (-1.50317, 0.57481), (-1.50317, -0.57481)}.\] Thus, due to reflection, we can obtain all the elements of $\ga$, and the corresponding distortion error is given by
\[\int\min_{a\in \ga} \|x-a\|^2 dP=2 (\te{distortion error due to }\tl p, \tl q, \tl r)=0.209898.\]

Comparing the distortion errors in Case~1 and Case~2, we see that the set $\ga$ in Case~1 forms the optimal set of six-means  (see Figure~\ref{Fig4}) with quantization error $V_6=0.198794$. Thus, the proof of the proposition is complete.
\end{proof}

\begin{prop} \label{prop72}
There are two different optimal sets of seven-means, one of them is
\begin{align*}
\set{(0, 0.98224), & (-1.26066, 0.741926), (-1.80513, -0.225347), (-0.68335, -0.916091), \\
&(0.68335, -0.916091), (1.80513, -0.225347), (1.26066, 0.741926)}
\end{align*}
 with quantization error $V_7=0.152179$.
\end{prop}
\begin{proof} Let $\ga$ be an optimal set of seven-means. Two cases can arise:

Case~1. $\ga$ contains a point from the major axis, three points from above the major axis, and the other three points from below the major axis.

Due to Remark~\ref{remark90}, in this case, we can assume that the three points in $\ga$ which are below the major axis are the reflections of the three points above the major axis.  Let $\ga$ contain a point, denoted $\tl p$, from the positive direction of the major axis. Let the points in $\ga$ which are above the major axis be $\tl q, \tl r$, and $\tl s$. Due to symmetry $(-2, 0)$ is a boundary point of the Voronoi regions. Let the boundary points of the Voronoi regions of $\tl q, \tl r$, and $\tl s$ be $\tl d, \tl e$, $\tl f$, and $(-2, 0)$, respectively, given by the parametric values $\gq=b, \gq=c, \gq=d$, and $\gq=\pi$. Then, we have
\begin{align*}
\tl d& =(2\cos b, \sin b), \ \tl e=(2 \cos c, \sin c), \ \tl f=(2\cos d, \sin d), \\
 \tl p&=\frac{\int_{2\pi-b}^{2\pi+b} \sqrt{4 \sin ^2\gq+\cos ^2\gq} (2 \cos \gq,\sin \gq)\, d\theta }{\int_{2\pi-b}^{2\pi+b} \sqrt{4 \sin ^2\gq+\cos ^2\gq} \, d\theta }, \ \tl q =\frac{\int_{b}^{c} \sqrt{4 \sin ^2\gq+\cos ^2\gq} (2 \cos \gq,\sin \gq) \, d\theta }{\int_{b}^{c} \sqrt{4 \sin ^2\gq+\cos ^2\gq} \ d\theta },\\
 \tl r& =\frac{\int_{c}^{d } \sqrt{4 \sin ^2\gq+\cos ^2\gq} (2 \cos \gq,\sin \gq) \, d\theta }{\int_{c}^{d} \sqrt{4 \sin ^2\gq+\cos ^2\gq} \, d\theta }, \te{ and } \tl s =\frac{\int_{d}^{\pi } \sqrt{4 \sin ^2\gq+\cos ^2\gq} (2 \cos \gq,\sin \gq) \, d\theta }{\int_{d}^{\pi} \sqrt{4 \sin ^2\gq+\cos ^2\gq} \, d\theta }.
\end{align*}
Solving the canonical equations $\rho(\tl d,\tl p)-\rho(\tl d, \tl q)=0$,  $\rho(\tl e, \tl q)-\rho(\tl e, \tl r)=0$, and $\rho(\tl f, \tl r)-\rho(\tl f, \tl s)=0$ we obtain $b=0.661475, c =1.42154$, and $d =2.11384$, which give the set $\set{\tl p, \tl q, \tl r, \tl s}$ equals
\[\set{(1.84197, 0), (0.953451, 0.852294), (-0.372194, 0.962574), (-1.6113, 0.517322)},\]
and the other three points in $\ga$ are the reflections of $\tl q, \tl r, \tl s$ with respect to the major axis. The corresponding distortion error is given by
\[\int\min_{a\in \ga} \|x-a\|^2 dP=\te{distortion error due to } \tl p +2 (\te{distortion error due to } \tl q, \tl r, \tl s)=0.152488.\]

Case~2. $\ga$ contains a point from the minor axis, three points are to the left of the minor axis, and the other three points are the reflections with respect to the minor axis.

Let $\ga$ contain a point, denoted by $\tl p$, from the positive direction of the minor axis. Due to symmetry $(0, -1)$ is a boundary point of the Voronoi regions. Let the points in $\ga$ which are to the left of the minor axis be $\tl q, \tl r$, and $\tl s$. Let the boundary points of their Voronoi regions be $\tl d, \tl e$, $\tl f$, and $(0, -1)$, respectively, given by the parametric values $\gq=b, \gq=c, \gq=d$, and $\gq=\frac {3\pi}2$. Then, we have
\begin{align*}
\tl d& =(2\cos b, \sin b), \ \tl e=(2 \cos c, \sin c), \ \tl f=(2\cos d, \sin d), \\
 \tl p&=\frac{\int_{\pi-b}^{b} \sqrt{4 \sin ^2\gq+\cos ^2\gq} (2 \cos \gq,\sin \gq)\, d\theta }{\int_{\pi-b}^{b} \sqrt{4 \sin ^2\gq+\cos ^2\gq} \, d\theta }, \ \tl q =\frac{\int_{b}^{c} \sqrt{4 \sin ^2\gq+\cos ^2\gq} (2 \cos \gq,\sin \gq) \, d\theta }{\int_{b}^{c} \sqrt{4 \sin ^2\gq+\cos ^2\gq} \ d\theta },\\
 \tl r& =\frac{\int_{c}^{d } \sqrt{4 \sin ^2\gq+\cos ^2\gq} (2 \cos \gq,\sin \gq) \, d\theta }{\int_{c}^{d} \sqrt{4 \sin ^2\gq+\cos ^2\gq} \, d\theta }, \te{ and } \tl s =\frac{\int_{d}^{\frac{3\pi}2 } \sqrt{4 \sin ^2\gq+\cos ^2\gq} (2 \cos \gq,\sin \gq) \, d\theta }{\int_{d}^{\frac {3\pi}2} \sqrt{4 \sin ^2\gq+\cos ^2\gq} \, d\theta }.
\end{align*}
Solving the canonical equations $\rho(\tl d,\tl p)-\rho(\tl d, \tl q)=0$,  $\rho(\tl e, \tl q)-\rho(\tl e, \tl r)=0$, and $\rho(\tl f, \tl r)-\rho(\tl f, \tl s)=0,$ we obtain $b=1.8999, c=2.71024$, and $d=3.97294$, which give the set $\set{\tl p, \tl q, \tl r, \tl s}$ equals
\[\set{(0, 0.98224), (-1.26066, 0.741926), (-1.80513, -0.225347), (-0.68335, -0.916091)},\]
and thus, due to reflection, we can obtain all the points in $\ga$, and the corresponding distortion error is given by
\[\int\min_{a\in \ga} \|x-a\|^2 dP=\te{distortion error due to } \tl p +2 (\te{distortion error due to } \tl q, \tl r, \tl s)=0.152179.\]

Comparing the distortion errors in Case~1 and Case~2, we see that the set $\ga$ in Case~2 forms the optimal set of seven-means with quantization error $V_7=0.152179$ (see Figure~\ref{Fig4}). Instead of choosing the point $\tl p$ from the positive direction of the minor axis, we can choose it from the negative direction of the minor axis. This will give another optimal set of seven-means. Thus, the proof of the proposition is complete.
\end{proof}

\begin{remark}
Following the technique given in Proposition~\ref{prop71}, we can obtain the optimal set of $n$-means for any even positive integer $n\geq 4$; on the other hand, if $n$ is odd, following the technique given in Proposition~\ref{prop72}, we can obtain the optimal sets of $n$-means for any odd positive integer $n\geq 3$. Notice that if $n$ is even, there are two points in the optimal set which lie on the major axis and so, the optimal set is unique; if $n$ is odd, there is one point in the optimal set which lies on the minor axis and so, there are two different optimal sets of $n$-means. Thus, we see that an optimal set of three-means is
$\set{(0, 0.943319), (-1.3166, -0.29241), (1.3166, -0.29241)}$ with quantization error
$V_3=0.661148$; the optimal set of four-means is
\[\set{(1.60986, 0), (0, 0.953301), (-1.60986, 0), (0, -0.953301)}\] with quantization error $V_4=0.393732$; and an optimal set of five-means is
\begin{align*}
\set{(0, 0.960639), (-1.669,    0.253562), & (-0.832677, -0.869617),(0.832677, -0.869617), \\
&(1.669,    0.253562)}
\end{align*}
 with quantization error $V_5=0.28329$ (see Figure~\ref{Fig4}).
\end{remark}

We now end the paper with the following conjecture.

\begin{conj} By Proposition~\ref{prop32}, we see that the quantization coefficient for the uniform distribution on a unit circle is $\frac {\pi^2}3$; on the other hand, by Proposition~\ref{prop55}, the quantization coefficient for the uniform distribution on the boundary of a regular hexagon inscribed in a unit circle is $3$. Notice that a regular $m$-sided polygon inscribed in a circle tends to the circle as $m$ tends to infinity. We conjecture that the quantization coefficient for the uniform distribution on the boundary of a regular $m$-sided polygon inscribed in a circle is an increasing function of $m$, and approaches to the quantization coefficient for the uniform distribution on the circle as $m$ tends to infinity.
\end{conj}
%
%\subsection*{Acknowledgement} The authors are grateful to the referees for their valuable comments and suggestions.

\end{document}